\documentclass{amsart}
\usepackage{amsfonts,amsthm,amsmath}
\usepackage{amssymb}
\usepackage{amscd}
\usepackage[utf8]{inputenc}
\usepackage[a4paper]{geometry}
\usepackage{enumerate}
\usepackage[usenames,dvipsnames]{color}
\usepackage{array}
\usepackage{array,tabularx}

\DeclareMathOperator{\Coeff}{Coeff}
\DeclareMathOperator*{\Res}{Res}
\newcommand\note[1]{\mbox{}\marginpar{ \scriptsize\raggedright
\hspace{1pt}\color{red} #1}}

\numberwithin{equation}{section}
\numberwithin{equation}{subsection}

\theoremstyle{plain}

\newtheorem{theorem}[equation]{Theorem}
\newtheorem{lemma}[equation]{Lemma}

\newtheorem{thm}[equation]{Theorem}
\newtheorem{cor}[equation]{Corollary}

\newtheorem*{theorem*}{Theorem}

\theoremstyle{definition}
\newtheorem{example}[equation]{Example}
\newtheorem{remark}[equation]{Remark}

\newtheorem{definition}[equation]{Definition}

\newtheorem{bekezdes}[equation]{}

\def\C{\mathbb C}
\def\Q{\mathbb Q}

\def\Z{\mathbb Z}

\def\im{{\rm Im}}

\newcommand{\caly}{{\mathcal Y}}
\newcommand{\calz}{{\mathcal Z}}
\newcommand{\calv}{{\mathcal V}}

\newcommand{\cali}{{\mathcal I}}

\newcommand{\calO}{{\mathcal O}}

\newcommand{\calS}{{\mathcal S}}

\newcommand{\calL}{\mathcal{L}}

\newcommand{\tX}{\widetilde{X}}

\newcommand{\cX}{{\mathcal X}}
\newcommand{\cO}{{\mathcal O}}
\newcommand{\bP}{{\mathbb P}}
\newcommand*{\linebundle}{\mathcal{L}}
\newcommand{\bC}{{\mathbb C}}
\newcommand{\cF}{{\mathcal F}}

\newcommand{\eca}{{\rm ECa}}
\newcommand{\pic}{{\rm Pic}}

\newcommand{\m}{\mathfrak{m}}\newcommand{\fr}{\mathfrak{r}}
\newcommand{\mfl}{\mathfrak{L}}

\newcommand{\bt}{{\mathbf t}}

\newcommand{\bZ}{{\mathbb{Z}}}
\newcommand{\bQ}{{\mathbb{Q}}}

\newcommand{\omegl}{\lambda}

\author{J\'anos Nagy}
\address{Rényi Alfréd Institute of Mathematics,  Budapest, Hungary}
\email{janomo4@gmail.com}

\title{The possible values of geometric genera of normal surface singularities}

\begin{document}



\begin{abstract}

In this article we prove that the possible geometric genuses $p_g(\tX)$ corresponding to normal surface singularities $\tX$ with fixed negative definite resolution graph $\mathcal{T}$ form an interval of integers.
Similarly let us have a resolution graph $\mathcal{T}$ and a fixed normal surface singularity $(X, 0)$ with resolution $\tX$ and resolution graph $\mathcal{T}$, furthermore consider a Chern class $l' \in L'$.
We prove that the possible values of $h^1(\tX, \calL)$, where $\calL \in \pic^{l'}(\tX)$ form an interval of integers.
\end{abstract}

\maketitle

\linespread{1.2}


\pagestyle{myheadings} \markboth{{\normalsize  J. Nagy}} {{\normalsize Line bundles}}


\section{Introduction}\label{s:intr}

In the theory of normal surface singularities in the last decade one of the major issues was to compare it's analytic invariants with it's topological invariants, which are computable just from
the resolution graph $\mathcal{T}$ of the given singularity, or equivalently from the link of the singularity, which is an oriented graph 3-manifold.

The main subject of research was to provide topological formulae for several discrete analytic invariants or at least topological candidates.
Of course, when we fix the topological type and vary the analytic structure most of these analytic invariants also can change, so we have a hope to find purely topological formulas just in the case of special analytical families.

One such large breakthrough was the previous work of Némethi \cite{NJEMS}, and Okuma \cite{Ok}, \cite{NOk}, that for splice quotient singularities, or in an another language for singularities satisfying the end curve condition, many analytic invariants like geometric genus or analytic Poincaré series coincide with their topological candidates like Seiberg-Witten invariants and topological multivariable Zeta function.

The coincidence of the Seiberg-Witten invariant and the geometric genus also holds in the case of Newton nondegenerate hypersurface singularities \cite{Bald1}.

There is an another topological candidate which is called path lattice cohomology, or in other articles $MIN_{\gamma}$, which is also an upper bound for the geometric genus for every
analytic structure, however for some special families of singularities there is equality, like in the weighted homogenous case, or Superisolated case \cite{Bald2} or also in the case of Newton nondegenerate hypersurface singularities \cite{Bald2}.

However there are resolution graphs, for which $MIN_{\gamma}$ is not the geometric genus of any singularity corresponding to it, so the upper bound is not sharp in general \cite{tcusp}.

For a cycle $l$ supported on the excptional divisor on a resolution of a normal surface singularity let us define the Riemann-Roch function $\chi(l) = h^0(\calO_l) - h^1(\calO_l)$.
This is topological, depending only on the coefficients of the cycle $l$ and on the resolution graph. Let $E$ denote the exceptional divisor of a resolution.

While the minimal possible geometric genus for a fixed topological type is already known, which is the geometric genus of a generic singularity and is given by the formula $1 - \min_{E \leq l \in L} \chi(l)$ in \cite{NNA2}, the determination of the maximal possible geometric genus is still an open problem.

The main massage of these results is that while the geometric genus can change when we vary the analytic structure, there are combinatorial candidates for this value, and equality happens
for special families of analytic types.

However we show in this article the following theorem:

\begin{theorem*}
Let $\mathcal{T}$ be a rational homology sphere resolution graph, then the possible geometric genuses of normal surface singularities with resolution graph $\mathcal{T}$ form an interval of integers $[1 - \min_{ E \leq l \in L} \chi(l), M]$.
\end{theorem*}

This result can be expected to be true in spite of the philosophy that for certain families of analytic structures the geometric genus chooses values among different topological candidates like Seiberg-Witten invariants or path lattice cohomology.

On the other hand these kinds of results were not the easiest to prove in the past without having a described moduli space of possible analytic structures and having just a few special analytic families with computable geometric genus, and without having more general families of singularities like generic singularities or relatively generic ones. 

\bigskip

We will heavily use the Abel map theory developed in \cite{NNA1} and \cite{NNA2}, which has got many similarities with the classical Abel map theory on smooth curves and Brill-Noether
theory, but many times the answers and even the questions are harder in our case.

\noindent

Notice that in the classical case, if one has a genus $g$ smooth curve $C$, then $h^1(\calO_C) = g$ independently of the analytic structure of $C$, so our question of 
the possible geometric genuses of normal surface singularities has got no counterpart in the classical case.

\noindent

On the other hand if one fixes a genus $g$ smooth curve $C$, and a degree $d$, then one has a stratificiation of the Picard group $\pic^{d}(C)$ by the Brill-Noether stratas
$W^r_d(C) = (\calL \in \pic^d(C) | h^0(C, \calL) \geq r+1)$.

\noindent

Several properties of this stratification, for example the dimension of the stratas, or their nonemptyness depends heavily on the analytic structure of $C$ and it is exactly the
task of the classical Brill-Noether theory to ivestigate its properties among the generic curves or among other special families of algebraic curves.

If $C$ is a smooth curve and $r > 0$ is a natural number such that $W^r_d(C)$ is nonempty then $W^{r-1}_d(C)$ is also nonempty, see \cite{ACGH}.

\noindent

This result in the classical Brill-Noether theory says basically that the possible values of $h^1$ of line bundles in $\pic^{d}(C)$ form an interval, and so it is very similar to our main theorem, in fact we will prove first the analouge of this statement in our case of normal surface singularities.

\noindent

Although it is a simple lemma in the classical case the proof is very far from working in our case, and we will need to use the methods and theorems from \cite{R} about relatively generic line bundles.

\noindent

Next we will prove our main theorem about the possible values of geometric genuses by following the usual philosophy developed in \cite{NNA1} and \cite{NNA2} that the cohomological behaviour of the line bundles on a fixed singularity $\tX$ is similar to the cohomological behaivour of the natural line bundles on analyitically varying but topologically fixed singularities.

\bigskip

\emph{More specially we will prove two main theorems:}

\bigskip

Let's have a resolution graph $\mathcal{T}$ and a corresponding normal surface singularity with resolution space $\tX$, furthermore let's fix a Chern class $l'$ and an effective
cycle $Z$. 

The first main theorem states that the possible values of $h^1(Z, \calL)$, where $\calL \in \pic^{l'}(Z)$ form an interval of integers:

\begin{theorem*}\textbf{A}
Let us have an arbitrary rational homology sphere resolution graph $\mathcal{T}$ and a corresponding singularity $(X, 0)$ with resolution $\tX$ and resolution graph $\mathcal{T}$, an effective cycle $Z$ and an abitrary Chern class $l'$. 

Let us denote $k = \max_{\calL \in \pic^{l'}(Z)} h^1(Z, \calL)$ and an arbitrary integer $r$ such that $ \chi(-l') - \min_{0 \leq l \leq Z} \chi(-l'+ l) \leq r \leq k$, then there is a line bundle $\calL \in \pic^{l'}(Z)$, such that $h^1(Z, \calL) = r$.
\end{theorem*}

\begin{remark}
Notice that by the formal neighborhood theorem,  if $Z$ is a large effective cycle supported on the exceptional divisor $E$, then $h^1(Z, \calL) = h^1(\tX, \calL)$, which means that the theorem 
above gives a similar theorem about the cohomology numbers $h^1(\tX, \calL)$.
\end{remark}

Similarly let us have a resolution graph $\mathcal{T}$, a Chern class $l'$ and an effective cycle $Z$, such that if $l' = \sum_{v \in \calv} b_v E_v$, then $b_v <0$ for
every vertex $v \in |Z|$ (Notice that the support $|Z|$ is not nessecarily the whole exceptional divisor $E$ which means that the condition is not equivalent to $l'$ being in the negative quadrant).

The second main theorem we prove states that the possible values the $h^1$ of natural line bundles $h^1(\calO_Z(l'))$ form an interval of integers if we consider any possible surface singularities with resolution graph $\mathcal{T}$:

\begin{theorem*}\textbf{B}
Let $\mathcal{T}$ be an arbitrary rational homology sphere resolution graph with vertex set $\calv$ and let $Z$ be an effective cycle on it, and suppose furthermore that $l' \in L'$ is a Chern class, such that $l' = \sum_{v \in \calv} a_v E_v^* \in L'$. Let us write $l' =  \sum_{v \in \calv} b_v E_v$ where $b_v$ are rational numbers, and assume that $b_v <0$ if $v \in |Z|$.

Suppose that we have a singularity $(X, 0)$ with resolution $\tX$ and resolution graph $\mathcal{T}$, and let us have the restricted natural line bundle $\calO_Z(l')$.

Suppose that $k = h^1(\calO_Z(l')) > \chi(-l') - \min_{0 \leq l \leq Z}  \chi(-l' + l)$, and let us have an arbitrary number $k > r \geq \chi(-l') - \min_{0 \leq l \leq Z}  \chi(-l' + l)$, then there is another singularity with resolution $\tX'$ and with resolution graph $\mathcal{T}$, for which one has $r = h^1(\calO_Z(l'))$.
\end{theorem*}
\begin{remark}
We need the condition that $b_v <0$ if $v \in |Z|$, because the cohomology numbers of the natural line bundle has the exact value $\chi(-l') - \min_{0 \leq l \leq Z}  \chi(-l' + l)$ only
in these cases form \cite{NNA2}, otherwise there are only much more complicated formulae for them from \cite{CNG}.
\end{remark}

Corollary\ref{geomgen} of Theorem\textbf{B} will show that the possible values of the geometric genuses $p_g(\tX)$ form an interval of integers, so Theorem \textbf{A}.

The structure of the article is the following:

In section 2) we summarise the necessary topological and analytic invariants of normal surface singularities and their main properties. In section 3) we recall the results about effective Cartier divisors and Abel maps on surface singularities from \cite{NNA1}.  In section 4) we recall Laufer's results about deformations of normal surface singularities. In section 5) we recall the results about cohomology of relatively generic line bundles and cohomology of natural line bundles on relatively generic surface singularities from \cite{R}.
 In section 6) we prove a key lemma about the behavior of $h^1$ of line bundles under certain blow ups and restrictions, which will be in the heart of every proof.
In section 7) we outline the proofs of Theorem\textbf{A} and Theorem\textbf{B} without going deeply into various technical details. Finally in section 8) we prove Theorem A and in section 9) we prove Theorem B and its corrolary about the possible geometric genuses of normal surface singularities.

\section{Preliminaries}\label{s:prel}

\subsection{The resolution}\label{ss:notation}
Let $(X,o)$ be the germ of a complex analytic normal surface singularity,
 and let us fix  a good resolution  $\phi:\widetilde{X}\to X$ of $(X,o)$.
We denote the exceptional curve $\phi^{-1}(0)$ by $E$, and let $\{E_v\}_{v\in\calv}$ be
its irreducible components. Set also $E_I:=\sum_{v\in I}E_v$ for any subset $I\subset \calv$.
For the cycle $l=\sum n_vE_v$ let its support be $|l|=\cup_{n_v\not=0}E_v$.
Mixing the two notations we will use $E_{|l|} = \sum_{v\in |l|}E_v$ for an arbitrary cycle $l$.
For more details see \cite{NCL,Nfive}.
\subsection{Topological invariants}\label{ss:topol}
Let $\Gamma$ be the dual resolution graph
associated with $\phi$;  it  is a connected graph.
Then $M:=\partial \widetilde{X}$, as a smooth oriented 3--manifold, 
 can be identified with the link of $(X,o)$, it is also
an oriented  plumbed 3--manifold associated with $\Gamma$.
{\it We will assume  (for any singularity we will deal with) that the link
 $M$ is a rational homology sphere,}
or, equivalently,  $\Gamma$ is a tree with all genus
decorations  zero. We use the same
notation $\mathcal{V}$ for the set of vertices.

The lattice $L:=H_2(\widetilde{X},\mathbb{Z})$ is  endowed
with a negative definite intersection form  $I=(\,,\,)$. It is
freely generated by the classes of 2--spheres $\{E_v\}_{v\in\mathcal{V}}$.
 The dual lattice $L':=H^2(\widetilde{X},\mathbb{Z})$ is generated
by the (anti)dual classes $\{E^*_v\}_{v\in\mathcal{V}}$ defined
by $(E^{*}_{v},E_{w})=-\delta_{vw}$, the opposite of the Kronecker symbol.
The intersection form embeds $L$ into $L'$. Then $H_1(M,\mathbb{Z})\simeq L'/L$, abridged by $H$.
Usually one also identifies $L'$ with those rational cycles $l'\in L\otimes \Q$ for which
$(l',L)\in\Z$ (or, $L'={\rm Hom}_\Z(L,\Z)\simeq H^2(\tX,\mathbb{Z})$), where the intersection form extends naturally.

All the $E_v$--coordinates of any $E^*_u$ are strict positive.
We define the Lipman cone as $\calS':=\{l'\in L'\,:\, (l', E_v)\leq 0 \ \mbox{for all $v$}\}$.
It is generated over $\bZ_{\geq 0}$ by $\{E^*_v\}_v$.
We also write $\calS:=\calS'\cap L$.

There is a natural partial ordering of $L'$ and $L$: we write $l_1'\geq l_2'$ if
$l_1'-l_2'=\sum _v r_vE_v$ with all $r_v\geq 0$. We set $L_{\geq 0}=\{l\in L\,:\, l\geq 0\}$ and
$L_{>0}=L_{\geq 0}\setminus \{0\}$.

We define the
  (anti)canonical cycle $Z_K\in L'$ via the {\it adjunction formulae}
$(-Z_K+E_v,E_v)+2=0$ for all $v\in \mathcal{V}$.
(In fact,  $Z_K=-c_1(\Omega^2_{\widetilde{X}})$, cf. (\ref{eq:PIC})).
In a minimal resolution $Z_K\in \calS'$.

Finally we consider the Riemann--Roch expression
 $\chi(l')=-(l',l'-Z_K)/2$ defined for any $l'\in L'$.

\subsection{Some analytic invariants}\label{ss:analinv}
{\bf The group ${\rm Pic}(\widetilde{X})$}
of  isomorphism classes of analytic line bundles on $\widetilde{X}$ appears in the (exponential) exact sequence
\begin{equation}\label{eq:PIC}
0\to {\rm Pic}^0(\widetilde{X})\to {\rm Pic}(\widetilde{X})\stackrel{c_1}
{\longrightarrow} L'\to 0, \end{equation}
where  $c_1$ denotes the first Chern class. Here
$ {\rm Pic}^0(\widetilde{X})=H^1(\widetilde{X},\calO_{\widetilde{X}})\simeq
\C^{p_g}$, where $p_g$ is the {\it geometric genus} of
$(X,o)$. $(X,o)$ is called {\it rational} if $p_g(X,o)=0$.
 Artin in \cite{Artin62,Artin66} characterized rationality topologically
via the graphs; such graphs are called `rational'. By this criterion, $\Gamma$
is rational if and only if $\chi(l)\geq 1$ for any effective non--zero cycle $l\in L_{>0}$.

The epimorphism
$c_1$ admits a unique group homomorphism section $l'\mapsto s(l')\in {\rm Pic}(\widetilde{X})$,
 which extends the natural
section $l\mapsto \calO_{\widetilde{X}}(l)$ valid for integral cycles $l\in L$, and
such that $c_1(s(l'))=l'$  \cite{OkumaRat}.
We call $s(l')$ the  {\it natural line bundles} on $\widetilde{X}$ and we denote $s(l') = \calO_{\tX}(l')$.
By  the very  definition, $\calL$ is natural if and only if some power $\calL^{\otimes n}$
of it has the form $\calO_{\tX}(l)$ for some $l\in L$.

\bekezdes $\mathbf{{Pic}(Z)}.$ \
Similarly, if $Z\in L_{>0}$ is a non--zero effective integral cycle such that its support is $|Z| =E$,
and $\calO_Z^*$ denotes
the sheaf of units of $\calO_Z$, then ${\rm Pic}(Z)=H^1(Z,\calO_Z^*)$ is  the group of isomorphism classes
of invertible sheaves on $Z$. It appears in the exact sequence
  \begin{equation}\label{eq:PICZ}
0\to {\rm Pic}^0(Z)\to {\rm Pic}(Z)\stackrel{c_1}
{\longrightarrow} L'\to 0, \end{equation}
where ${\rm Pic}^0(Z)=H^1(Z,\calO_Z)$.
If $Z_2\geq Z_1$ then there are natural restriction maps,
${\rm Pic}(\widetilde{X})\to {\rm Pic}(Z_2)\to {\rm Pic}(Z_1)$.
Similar restrictions are defined at  ${\rm Pic}^0$ level too.
These restrictions are homomorphisms of the exact sequences  (\ref{eq:PIC}) and (\ref{eq:PICZ}).

Furthermore, we define a section of (\ref{eq:PICZ}) by
$s_Z(l'):=
{\mathcal O}_{\widetilde{X}}(l')|_{Z}$, they also satisfy $c_1\circ s_Z={\rm id}_{L'}$. 
We write  ${\mathcal O}_{Z}(l')$ for $s_Z(l')$, and we call them
 {\it natural line bundles } on $Z$.

We also use the notations ${\rm Pic}^{l'}(\widetilde{X}):=c_1^{-1}(l')
\subset {\rm Pic}(\widetilde{X})$ and
${\rm Pic}^{l'}(Z):=c_1^{-1}(l')\subset{\rm Pic}(Z)$
respectively. Multiplication by $\calO_{\widetilde{X}}(-l')$, or by
$\calO_Z(-l')$, provides natural affine--space isomorphisms
${\rm Pic}^{l'}(\widetilde{X})\to {\rm Pic}^0(\widetilde{X})$ and
${\rm Pic}^{l'}(Z)\to {\rm Pic}^0(Z)$.

\bekezdes\label{bek:restrnlb} {\bf Restricted natural line bundles.}
The following warning is appropriate.
Note that if $\tX_1$ is a connected small convenient  neighbourhood
of the union of some of the exceptional divisors (hence $\tX_1$ also stays as the resolution
of the singularity obtained by contraction of that union of exceptional  curves), then one can repeat the definition of
natural line bundles at the level of $\tX_1$ as well (as a splitting of (\ref{eq:PIC}) applied for
$\tX_1$). However, the restriction to
$\tX_1$ of a natural line bundle of $\tX$ (even of type
$\calO_{\tX}(l)$ with $l$ integral cycle supported on $E$)  is usually not natural on $\tX_1$:
$\calO_{\tX}(l')|_{\tX_1}\not= \calO_{\tX_1}(R(l'))$
 (where $R:H^2(\tX,\Z)\to H^2(\tX_1,\Z)$ is the natural cohomological 
 restriction), though their Chern classes coincide.

Therefore, in inductive procedure when such restriction is needed,
 we will deal with the family of {\it restricted natural line bundles}. This means the following.
If we have two resolution spaces $\tX_1 \subset \tX$ with resolution graphs $\mathcal{T}_1 \subset \mathcal{T}$ and we have a Chern class $l' \in L'$, then we denote 
by $\calO_{\tX_1}(l') = \calO_{\tX}(l') | \tX_1$ the restriction of the natural line bundle $\calO_{\tX}(l')$.
Similarly if $Z$ is an effective integer cycle on $\tX$ with maybe $|Z| \neq E$, then we denote $\calO_{Z}(l') = \calO_{\tX}(l') | Z$.

Furthermore if $\calL$ is a line bundle on $\tX_1$, then we denote $\calL(l') = \calL \otimes \calO_{\tX}(l')$.
Similarly if $Z$ is  an effective integer cycle on $\tX$ and $\calL$ is a line bundle on $Z$, then we denote $\calL(l') = \calL \otimes \calO_Z(l')$.

\bekezdes \label{bek:ansemgr} {\bf The analytic semigroups.} \
By definition, the analytic semigroup associated with the resolution $\tX$ is

\begin{equation}\label{eq:ansemgr}
\calS'_{an}:= \{l'\in L' \,:\,\calO_{\tX}(-l')\ \mbox{has no  fixed components}\}.
\end{equation}
It is a subsemigroup of $\calS'$. One also sets $\calS_{an}:=\calS_{an}'\cap L$, a subsemigroup
of $\calS$. In fact, $\calS_{an}$
consists of the restrictions   ${\rm div}_E(f)$ of the divisors
${\rm div}(f\circ \phi)$ to $E$, where $f$ runs over $\calO_{X,o}$. Therefore, if $s_1, s_2\in \calS_{an}$, then
${\rm min}\{s_1,s_2\}\in \calS_{an}$ as well (take the generic linear combination of the corresponding functions).
In particular,  for any $l\in L$, there exists a {\it unique} minimal
$s\in \calS_{an}$ with $s\geq l$.

Similarly, for any $h\in H=L'/L$ set $\calS'_{an,h}:\{l'\in \calS_{an}\,:\, [l']=h\}$.
Then for any  $s'_1, s'_2\in \calS_{an,h}$ one has
${\rm min}\{s'_1,s'_2\}\in \calS_{an,h}$, and so
for any $l'\in L'$   there exists a unique minimal
$s'\in \calS_{an,[l']}$ with $s'\geq l'$.

\subsection{Notations.} We will write $Z_{min}\in L$ for the  {\it minimal} (or fundamental, or Artin) cycle, which is
the minimal non--zero cycle of $\calS'\cap L$ \cite{Artin62,Artin66}. Yau's {\it maximal ideal cycle}
$Z_{max}\in L$ defines the  divisorial part of the pullback of the maximal ideal $\m_{X,o}\subset \calO_{X,o}$, i.e.
 $\phi^*{\m_{X,o}}\cdot \calO_{\widetilde{X}}=\calO_{\widetilde{X}}(-Z_{max})\cdot \cali$,
where $\cali$ is an ideal sheaf with 0--dimensional support \cite{Yau1}. In general $Z_{min}\leq Z_{max}$.

\section{Effective Cartier divisors and Abel maps}

  In this section we review some needed material from \cite{NNA1}.

We fix a good resolution $\phi:\tX\to X$ of a normal surface singularity,
whose link is a rational homology sphere. 

\subsection{} \label{ss:4.1}
Let us fix an effective integral cycle  $Z\in L$, $Z\geq E$. (The restriction $Z\geq E$ is imposed by the
easement of the presentation, everything can be adopted  for $Z>0$).

Let $\eca(Z)$  be the space of effective Cartier (zero dimensional) divisors supported on  $Z$.
Taking the class of a Cartier divisor provides  a map
$c:\eca(Z)\to \pic(Z)$.
Let  $\eca^{l'}(Z)$ be the set of effective Cartier divisors with
Chern class $l'\in L'$, that is,
$\eca^{l'}(Z):=c^{-1}(\pic^{l'}(Z))$.

We consider the restriction of $c$, $c^{l'}:\eca^{l'}(Z)
\to \pic^{l'}(Z)$ too, sometimes still denoted by $c$.

As usual, we say that $\calL\in \pic^{l'}(Z)$ has no fixed components if
\begin{equation}\label{eq:H_0}
H^0(Z,\calL)_{reg}:=H^0(Z,\calL)\setminus \bigcup_v H^0(Z-E_v, \calL(-E_v))
\end{equation}
is non--empty.
Note that $H^0(Z,\calL)$ is a module over the algebra
$H^0(\calO_Z)$, hence one has a natural action of $H^0(\calO_Z^*)$ on
$H^0(Z, \calL)_{reg}$. This second action is algebraic and free.  Furthermore,
 $\calL\in \pic^{l'}(Z)$ is in the image of $c$ if and only if
$H^0(Z,\calL)_{reg}\not=\emptyset$. In this case, $c^{-1}(\calL)=H^0(Z,\calL)_{reg}/H^0(\calO_Z^*)$.

One verifies that $\eca^{l'}(Z)\not=\emptyset$ if and only if $-l'\in \calS'\setminus \{0\}$. Therefore, it is convenient to modify the definition of $\eca$ in the case $l'=0$: we (re)define $\eca^0(Z)=\{\emptyset\}$,
as the one--element set consisting of the `empty divisor'. We also take $c^0(\emptyset):=\calO_Z$, then we have
\begin{equation}\label{eq:empty}
\eca^{l'}(Z)\not =\emptyset \ \ \Leftrightarrow \ \ l'\in -\calS'.
\end{equation}
If $l'\in -\calS'$  then
  $\eca^{l'}(Z)$ is a smooth variety of dimension $(l',Z)$. Moreover,
if $\calL\in \im (c^{l'}(Z))$ (the image of the map $c^{l'}$)
then  the fiber $c^{-1}(\calL)$
 is a smooth, irreducible quasiprojective variety of  dimension
 \begin{equation}\label{eq:dimfiber}
\dim(c^{-1}(\calL))= h^0(Z,\calL)-h^0(\calO_Z)=
 (l',Z)+h^1(Z,\calL)-h^1(\calO_Z).
 \end{equation}

Consider again  a Chern class $l'\in-\calS'$ as above.
The $E^*$--support $I(l')\subset \calv$ of $l'$ is defined via the identity  $l'=\sum_{v\in I(l')}a_vE^*_v$ with all
$\{a_v\}_{v\in I}$ nonzero. Its role is the following.

Besides the Abel map $c^{l'}(Z)$ one can consider its `multiples' $\{c^{nl'}(Z)\}_{n\geq 1}$ as well. It turns out
(cf. \cite[\S 6]{NNA1}) that $n\mapsto \dim \im (c^{nl'}(Z))$
is a non-decreasing sequence, and  $\im (c^{nl'}(Z))$ is an affine subspace for $n\gg 1$, whose dimension $e_Z(l')$ is independent of $n\gg 0$, and essentially it depends only
on $I(l')$.
We denote the linearisation of this affine subspace by $V_Z(I) \subset H^1(\calO_Z)$ or if the cycle $Z \gg 0$, then $ V_{\tX}(I) \subset H^1(\calO_{\tX})$.

Moreover, by \cite[Theorem 6.1.9]{NNA1},
\begin{equation*}\label{eq:ezl}
e_Z(l')=h^1(\calO_Z)-h^1(\calO_{Z|_{\calv\setminus I(l')}}),
\end{equation*}
where $Z|_{\calv\setminus I(l')}$ is the restriction of the cycle $Z$ to its $\{E_v\}_{v\in \calv\setminus I(l')}$
coordinates.

If $Z\gg 0$ (i.e. all its $E_v$--coordinated are very large), then (\ref{eq:ezl}) reads as
\begin{equation*}\label{eq:ezlb}
e_Z(l')=h^1(\calO_{\tX})-h^1(\calO_{\tX(\calv\setminus I(l'))}),
\end{equation*}
where $\tX(\calv\setminus I(l'))$ is a convenient small tubular neighbourhood of $\cup_{v\in \calv\setminus I(l')}E_v$.

Let $\Omega _{\tX}(I)$ be the subspace of $H^0(\tX\setminus E, \Omega^2_{\tX})/ H^0(\tX,\Omega_{\tX}^2)$ generated by differential forms which have no poles along $E_I\setminus \cup_{v\not\in I}E_v$.
Then, cf. \cite[\S8]{NNA1},
\begin{equation*}\label{eq:ezlc}
h^1(\calO_{\tX(\calv\setminus I)})=\dim \Omega_{\tX}(I).
\end{equation*}

Similarly let $\Omega _{Z}(I)$ be the subspace of $H^0(\calO_{\tX}(K + Z))/ H^0(\calO_{\tX}(K))$ generated by differential forms which have no poles along $E_I\setminus \cup_{v\not\in I}E_v$.
Then, cf. \cite[\S8]{NNA1},
\begin{equation*}\label{eq:ezlc}
h^1(\calO_{Z_{(\calv\setminus I)}})=\dim \Omega_{Z}(I).
\end{equation*}

We have also the following duality from \cite{NNA1} supporting the equalities above:

\begin{theorem}\cite{NNA1}\label{th:DUALVO}
Via Laufer duality one has  $V_{\tX}(I)^*=\Omega_{\tX}(I)$ and $V_{Z}(I)^*=\Omega_Z(I)$.
\end{theorem}

\section{Laufer's results}\label{ss:LauferDef}
In this section we review some results of Laufer regarding deformations of the analytic structure on a resolution space of a normal surface singularity with fixed resolution graph
(and deformations of non--reduced analytic spaces supported on  exceptional curves) \cite{LaDef1}.

First, let us fix a normal surface singularity $(X,o)$ and a good resolution $\phi:(\tX,E)\to (X,o)$
with reduced exceptional curve $E=\phi^{-1}(o)$, whose irreducible decomposition is $\cup_{v\in\calv}E_v$ and dual graph $\Gamma$.
Let $\cali_v$ be the ideal sheaf of $E_v\subset \tX$. Then for arbitrary
 positive integers $\{r_v\}_{v\in \calv}$ one defines two objects, an analytic one and a topological (combinatorial) one.
 At analytic level, one sets  the ideal sheaf  $\cali(r):=\prod_v \cali_v^{r_v}$
 and the non--reduces space $\calO_{Z(r)}:=\calO_{\tX}/\cali(r)$ supported on $E$.

  The topological object is a graph with multiplicities, denoted by $\Gamma(r)$. As a non--decorated graph coincides with the graph $\Gamma$ without decorations. Additionally each vertex $v$ has a
  `multiplicity decoration' $r_v$, and we put also the self--intersection  decoration $E_v^2$
  whenever $r_v>1$. (Hence, the vertex $v$ does not inherit the self--intersection decoration
  of $v$ if $r_v=1$).
 Note that the  abstract 1--dimensional analytic space $Z(r)$ determines by its reduced structure
 the shape of the dual graph $\Gamma$, and by its non--reduced structure
 all the multiplicities $\{r_v\}_{v\in\calv}$, and additionally,
 all the self--intersection numbers $E_v^2$ for those $v$'s when  $r_v>1$
 (see \cite[Lemma 3.1]{LaDef1}).

 We say that the space $Z(r)$ has topological type $\Gamma(r)$.

Clearly, the analytic structure of $(X,o)$, hence of $\tX$ too, determines each 1--dimensional
non--reduced space $\calO_{Z(r)}$.  The converse is also true in the following sense.
\begin{theorem} \ \cite[Th. 6.20]{Lauferbook},\cite[Prop. 3.8]{LaDef1}
(a) Consider an abstract 1--dimensional space $\calO_{Z(r)}$, whose topological type
$\Gamma(r)$ can be completed to a negative definite graph $\Gamma$ (or, lattice $L$).
Then there exists a 2--dimensional
manifold $\tX$ in which $Z(r)$ can be embedded with support $E$
such that the intersection matrix inherited from the embedding $E\subset \tX$ is the
negative definite lattice $L$.
In particular (since by Grauert theorem the exceptional locus $E$ in $\tX$ can be contracted to a normal singularity),
any such $Z(r)$ is always associated with a normal surface singularity (as above).

(b)  Suppose that we have two singularities $(X,o)$ and $(X',o)$ with good resolutions as above with the
same resolution graph $\Gamma$. Depending solely on $\Gamma$ the integers $\{r_v\}_v$ may be chosen so
large that if $\calO_{Z(r)}\simeq \calO_{Z'(r)}$, then $E\subset \tX$ and $E'\subset \tX'$ have
biholomorphically equivalent neighbourhoods via a map taking $E$ to $E'$.
(For a concrete estimate how large $r$ should be see Theorem 6.20 in \cite{Lauferbook}.)
\end{theorem}
In particular, in the deformation theory of $\tX$ it is enough to consider the deformations of
non--reduced spaces of type $\calO_{Z(r)}$.

Fix a non--reduced 1--dimensional space $Z=Z(r)$ with topological type $\Gamma(r)$.
Following Laufer let's choose also a closed subspace $Y$ of $Z$ (whose support can be smaller, it can be even empty).
More precisely, $(Z,Y)$ locally is isomorphic with $(\C\{x,y\}/(x^ay^b),\C\{x,y\}/(x^cy^d))$,
where $a\geq c\geq 0$, $b\geq d\geq 0$, $a>0$.
  The ideal of $Y$ in $\calO_Z$ is denoted by $\cali_Y$.

\begin{definition}\label{def:1}\ \cite[Def. 2.1]{LaDef1}
 A deformation of $Z$, fixing $Y$, consists of the following data:

(i) There  exists an analytic space $\calz$ and a proper map $\omegl:\calz\to Q$, where
$Q$ is a manifold containing a distinguished point $0$.

(ii) Over a point $q\in Q$ the fiber $Z_q$ is the subspace of $\calz$ determined by the ideal
sheaf $\omegl^* (\mathfrak{m}_q)$ (where $\mathfrak{m}_q$ is the maximal ideal of $q$). $Z$
is isomorphic with $Z_0$, usually they are  identified.

(iii) $\omegl$ is a trivial deformation of $Y$ (that is, there is a closed subspace
$\caly\subset \calz$ and the restriction of $\omegl$ to $\caly$ is a trivial deformation of $Y$).

(iv) $\omegl$ is {\it locally trivial} in a way which extends the trivial deformation $\omegl|_{\caly}$.
This means  that for ant $q\in Q$ and $z\in \calz$ there exist a neighborhood $W$ of $z$ in $\calz$,
a neighborhood $V$ of $z$ in $Z_q$, a neighborhood $U$ of $q$ in $Q$, and an isomorphism
$\phi:W\to V\times U$ such that $\omegl|_W=pr_2\circ \phi$ (compatibly with the trivialization
of $\caly$ from (iii)), where $pr_2$ is the second projection; for more see  [loc.cit.].
\end{definition}
One verifies that under deformations (with connected base space) the topological type of the fibers
$Z_q$,  namely $\Gamma(r)$, stays constant (see \cite[Lemma 3.1]{LaDef1}).

\begin{definition}\label{def:2}\ \cite[Def. 2.4]{LaDef1}
A deformation $\omegl:\calz\to Q$ of $Z$, fixing $Y$, is complete at $0$ if, given any deformation
$\tau:{\mathcal P}\to R$ of $Z$ fixing $Y$, there is a neighbourhood $R'$ of $0$ in $R$  and a
 holomorphic map $f:R'\to Q$ such that $\tau$ restricted to $\tau^{-1}(R')$ is the deformation
$f^*\omegl$. Furthermore, $\omegl$ is complete if it is complete at each point $q\in Q$.
\end{definition}
Laufer proved the following  results.
\begin{theorem}\label{th:La2}\ \cite[Theorems 2.1, 2.3, 3.4, 3.6]{LaDef1}
Let $\theta_{Z,Y}={\mathcal Hom}_{Z}(\Omega^1_Z,\cali_Y)$ be the sheaf of germs of vector fields on $Z$  which vanish on $Y$, and let $\omegl :\calz\to Q$ be a deformation of $Z$, fixing $Y$.

 (a) If the Kodaira--Spencer map
 $\rho_0:T_0Q\to H^1(Z,\theta_{Z,Y})$ is surjective, then $\omegl$ is complete at $0$.

 (b) If $\rho_0$ is surjective than $\rho_q$ is surjective for all $q$ sufficiently near to $0$.

 (c)  There exists a deformation $\omegl$ with $\rho_0$ bijective. In such a case in a neighbourhood $U$ of $0$ the deformation is essentially unique, and  the fiber above $q$ is isomorphic to $Z$
 for only at most countably many $q$ in $U$.
\end{theorem}

\bekezdes\label{bek:funk} {\bf Functoriality} 

Let $Z'$ be a closed subspace of $Z$ such that
$\cali_{Z'}\subset \cali_Y\subset \calO_Z$. Then there is a natural reduction of pairs
$(\calO_Z,\calO_Y)\to (\calO_{Z'},\calO_Y)$. Hence, any deformation $\omegl:\calz\to Q$
of $Z$ fixing $Y$ reduces to a deformation $\omegl':\calz'\to Q$
of $Z'$ fixing $Y$. Furthermore, of $\omegl$ is complete then $\omegl'$ is automatically
complete as well (since $H^1(Z,\theta_{Z,Y})\to H^1(Z',\theta_{Z',Y})$ is onto).

\section{Relatively generic analytic structures on surface singularities}

In this section we wish to summarise the results from \cite{R} about relatively generic analytic structures what we need in this article, as always we again deal only with rational homology sphere
resolution graphs.

We consider an effective integer cycle $Z \geq E$ on a resolution $\tX$ with resolution graph $\mathcal{T}$, and a smaller cycle $Z_1 \leq Z$, where we denote $|Z_1| = \calv_1$ and the possibly nonconnected subgraph corresponding to it by $\mathcal{T}_1$.

We have the restriction map $r: \pic(Z)\to \pic(Z_1)$ and one has also the (cohomological) restriction operator
  $R_1 : L'(\mathcal{T}) \to L_1':=L'(\mathcal{T}_1)$
(defined as $R_1(E^*_v(\mathcal{T}))=E^*_v(\mathcal{T}_1)$ if $v\in \calv_1$, and
$R_1(E^*_v(\mathcal{T}))=0$ otherwise).

For any $\calL\in \pic(Z)$ and any $l'\in L'(\mathcal{T})$ it satisfies
\begin{equation*}
c_1(r(\calL))=R_1(c_1(\calL)).
\end{equation*}

In particular, we have the following commutative diagram as well:

\begin{equation*}  
\begin{picture}(200,40)(30,0)
\put(50,37){\makebox(0,0)[l]{$
\ \ \eca^{l'}(Z)\ \ \ \ \ \stackrel{c^{l'}(Z)}{\longrightarrow} \ \ \ \pic^{l'}(Z)$}}
\put(50,8){\makebox(0,0)[l]{$
\eca^{R_1(l')}(Z_1)\ \ \stackrel{c^{R_1(l')}(Z_1)}{\longrightarrow} \  \pic^{R_1(l')}(Z_1)$}}
\put(162,22){\makebox(0,0){$\downarrow \, $\tiny{$r$}}}
\put(78,22){\makebox(0,0){$\downarrow \, $\tiny{$\fr$}}}
\end{picture}
\end{equation*}

We can consider instead of the `total' Abel map $c^{l'}(Z)$ only its restriction above a fixed fiber of $r$.

That is, we fix some  $\mfl\in \pic^{R_1(l')}(Z_1)$, and we study the restriction of $c^{l'}(Z)$ to $(r\circ c^{l'}(Z))^{-1}(\mfl)\to r^{-1}(\mfl)$.

 The subvariety $(r\circ c^{l'}(Z))^{-1}(\mfl)
=(c^{R_1(l')}(Z_1) \circ \fr)^{-1}(\mfl) \subset \eca^{l'}(Z)$ is denoted by $\eca^{l', \mfl}(Z)$.

\begin{theorem}\cite{R}\label{relativspace}
Fix an arbitrary singularity $\tX$ a Chern class $l'\in -\calS'$, an integer effective cycle $Z\geq E$ and a cycle $Z_1 \leq Z$ and let's have a line bundle  $\mfl\in \pic^{R(l')}(Z_1)$.
Assume that  $\eca^{l', \mfl}(Z)$ is nonempty, then it is smooth of dimension $h^1(Z_1,\mfl)  - h^1(\calO_{Z_1})+ (l', Z)$ and irreducible.
\end{theorem}

Let's recall from \cite{R} the analouge of the theroems about dominance of Abel maps in the relative setup:

\begin{definition}\cite{R}
Fix an arbitrary singularity $\tX$, a Chern class $l'\in -\calS'$, an integer effective cycle $Z\geq E$, a cycle $Z_1 \leq Z$ and a line bundle $\mfl\in \pic^{R_1(l')}(Z_1)$ as above.
We say that the pair $(l',\mfl ) $ is {\it relative
dominant} on the cycle $Z$, if the closure of $ r^{-1}(\mfl)\cap \im(c^{l'}(Z))$ is $r^{-1}(\mfl)$.
\end{definition}

\begin{theorem}\label{th:dominantrel}\cite{R}
 One has the following facts:

(1) If $(l',\mfl)$ is relative dominant on the cycle $Z$, then $ \eca^{l', \mfl}(Z)$ is
nonempty and $h^1(Z,\calL)= h^1(Z_1,\mfl)$ for any
generic line bundle $\calL\in r^{-1}(\mfl)$.

(2) $(l',\mfl)$ is relative dominant on the cycle $Z$,  if and only if for all
 $0<l\leq Z$, $l\in L$ one has
$$\chi(-l')- h^1(Z_1, \mfl) < \chi(-l'+l)-
 h^1((Z-l)_1, \mfl(-l)).$$, where we denote $(Z-l)_1 = \min(Z-l, Z_1)$.
\end{theorem}

\begin{theorem}\label{th:hegy2rel}\cite{R}
Fix an arbitrary singularity $\tX$, a Chern class $l'\in -\calS'$, an integer effective cycle $Z\geq E$, a cycle $Z_1 \leq Z$ and a line bundle $\mfl\in \pic^{R_1(l')}(Z_1)$ as in Theorem \ref{th:dominantrel}. 
Then for any $\calL\in r^{-1}(\mfl)$ one has
\begin{equation*}\label{eq:genericLrel}
\begin{array}{ll}h^1(Z,\calL)\geq \chi(-l')-
\min_{0\leq l\leq Z,\ l\in L} \{\,\chi(-l'+l) -
h^1((Z-l)_1, \mfl(-l))\, \}, \ \ \mbox{or equivalently,}\\
h^0(Z,\calL)\geq \max_{0\leq l\leq Z,\, l\in L}
\{\,\chi(Z-l,\calL(-l))+  h^1((Z-l)_1, \mfl(-l))\,\}.\end{array}\end{equation*}
Furthermore, if $\calL$ is generic in $r^{-1}(\mfl)$
then in both inequalities we have equalities.
\end{theorem}

\subsection{Relatively generic analytic structures}

In the following we recall the results from \cite{R} about relatively generic analytic structures:

Let's fix a a topological type, in other words a resolution graph $\mathcal{T}$ with vertex set $\calv$,
we consider a partition $\calv = \calv_1 \cup  \calv_2$.

They define two (not necessarily connected) subgraphs $\mathcal{T}_1$ and $\mathcal{T}_2$.

We call the intersection of an exceptional divisor from
$ \calv_1 $ with an exceptional divisor from  $ \calv_2 $ a
{\it contact point}.

 For any $Z\in L=L(\mathcal{T})$ we write $Z=Z_1+Z_2$,
where $Z_i\in L(\mathcal{T}_i)$ is
supported in $\mathcal{T}_i$ ($i=1,2$).
Furthermore, parallel to the restrictions
$r_i : \pic(Z)\to \pic(Z_i)$ one also has the (cohomological) restriction operators
  $R_i : L'(\mathcal{T}) \to L_i':=L'(\mathcal{T}_i)$
(defined as $R_i(E^*_v(\mathcal{T}))=E^*_v(\mathcal{T}_i)$ if $v\in \calv_i$, and
$R_i(E^*_v(\mathcal{T}))=0$ otherwise).

For any $l'\in L'(\mathcal{T})$ and any $\calL\in \pic^{l'}(Z)$ it satisfies $c_1(r_i(\calL))=R_i(c_1(\calL))$.

In the following for the sake of simplicity we will denote $r = r_1$ and $R = R_1$.

Furthermore let's have a fixed analytic type $\tX_1$ for the subgraph $\mathcal{T}_1$ (if $\mathcal{T}_1$ is disconnected, then an analytic type for each of its connected components).

Also for each vertex $v_2 \in \calv_2$ which has got a neighbour $v_1$ in $\calv_1$ we fix a cut $D_{v_2}$ on $\tX_1$, along we glue the exceptional divisor $E_{v_2}$.
This means that $D_{v_2}$ is a divisor, which intersects the exceptional divisor $E_{v_1}$ transversally in one point and we will glue the tubular neighborhood of the exceptional divisor $E_{v_2}$ in a way, such that $E_{v_2} \cap \tX_1 = D_{v_2}$.

If for some vertex $v_2 \in \calv_2$, which has got a neighbour in $\calv_1$ we don't say explicitely what is the fixed cut, then it should be understood in the way that we glue the exceptional divisor $E_{v_2}$ along a generic cut.

Let's glue the tubular neihgbourhoods of the exceptional divisors $E_{v_2}, v_2 \in \calv_2$  with the above conditions generically to the fixed resolution $\tX_1$.
We get a singularity $\tX$ with resolution graph $\mathcal{T}$ and we say that $\tX$ is a relatively generic singularity corresponding to the analytical structure $\tX_1$ and the cuts $D_{v_2}$. 

For the more precise explanation of relative genericity look at \cite{R}.

We have the following theorems with this setup from \cite{R}:

\begin{theorem}\cite{R}\label{relgen1}
Let's have the setup as above and let's have furthermore an effective cycle $Z$ on $\tX$ and let's have $Z = Z_1 + Z_2$, where $|Z_1| \subset \calv_1$ and $|Z_2| \subset \calv_2$.

Let's have the natural line bundle $\calL= \calO_{\tX}(l')$ on $\tX$, such that $ l' = - \sum_{v \in \calv} a_v E_v \in L'_{\mathcal{T}}$, with $a_v > 0, v \in \calv_2 \cap |Z|$,  and let's denote $c_1 (\calL | Z) = l'_ Z \in L'_{|Z|}$, furthermore let's denote $\mfl = \calL | Z_1$, then we have the following:

We have $H^0(Z,\calL)_{reg} \not=\emptyset$ if and only if $(l',\mfl)$ is relative dominant on the cycle $Z$ or equivalently:

\begin{equation*}
\chi(-l')- h^1(Z_1, \mfl) < \chi(-l'+l)-  h^1((Z-l)_1, \mfl(-l)),
\end{equation*}
for all $0 < l \leq Z$.
\end{theorem}

\begin{theorem}\cite{R}\label{relgen2}
Let's have the same setup as in the previous theorem, then we have:

\begin{equation*}
h^1(Z, \calL) =  h^1(Z, \calL_{gen}),                                   
\end{equation*}
where $\calL_{gen}$ is a generic line bundle in $ r^{-1}(\mfl) \subset \pic^{l'_m}(Z)$, or equivalently:

\begin{equation*}
h^1(Z, \calL)= \chi(-l') - \min_{0 \leq l \leq Z}(\chi(-l'+l)-  h^1((Z-l)_1, \mfl(-l))).
\end{equation*}
\end{theorem}

\begin{remark}

In the theorems above in any formula one can replace $l'$ with $l'_Z$, since for every $0 \leq l \leq Z$ one has $\chi(-l')- \chi(-l'+l) = \chi(-l'_Z)- \chi(-l'_Z+l) = -(l', l) - \chi(l)$.
\end{remark}

\section{Reduction lemma of $h^1$ by $1$.}

We prove the following key lemma which will be useful in the following proofs:

\begin{lemma}\label{minus}
Let's have a rational homology sphere resolution graph $\mathcal{T}$, a corresponding singularity with resolution $\tX$ and an effective cycle $Z$ on it, for which we have $H^0(\calO_Z(K + Z))_{reg} \neq \emptyset$, or equivalently $h^1(\calO_Z) > h^1(\calO_{Z'})$ for every cycle $0 \leq Z' < Z$.

Let's have furthermore a line bundle $\calL$ on the cycle $Z$, such that $H^0(Z, \calL)_{reg} \neq \emptyset$ and $h^1(Z', \calL) < h^1(Z, \calL)$ for every cyle $0 \leq Z' < Z$.

Assume that we have a vertex $u \in |Z|$, such that $Z_u \geq 2$ and let's blow up the exceptional divisor $E_u$ sequentially along generic points $Z_v - 1$ times and let's denote the last $-1$-curve we get by $E_{u'}$. 

Let's denote the new vertex set by $\calv_{new}$, and the new resolution by $\tX_{new}$.

Let's denote furthermore $Z_{new} = \pi^*(Z) - \sum_{v \in \calv_{new} \setminus \calv} d(v) E_v$, where $d(v)$ is the distance of $v$ from the vertex set $\calv$, where we have $h^1(\calO_{Z_{new}}) = h^1(\calO_Z)$. 

Let's denote furthemore  $Z_r = Z_{new} - E_{u'}$, with these notations we have $h^1(Z_r, \pi^*(\calL)) = h^1(Z, \calL) - 1$.
\end{lemma}

\begin{proof}

Notice that every differential form in $H^1(\calO_Z)^* \cong \frac{H^0(\calO_{\tX}(K+Z'))}{H^0(\calO_{\tX}(K))}$ has got a pole on the exceptional divisor $E_{u'}$ of order at most $1$, since the order of a differential form decreases at every blow up with at least $1$.

Let's have a section $s \in H^0(Z, \calL)_{reg}$ and let's denote $|s| = D$.

Now let's have the line bundle $\calL_{new} = \pi^*(\calL)$ on the cycle $Z_{new}$, we  have $h^1(Z_{new}, \calL_{new}) = h^1(Z, \calL) = k$ and there is a section $s_{new} \in H^0(Z'_{new},\calL_{new})_{reg}$, such that $|s_{new}| = D$.

We claim first that $h^1(\calO_{Z_{r}}) < h^1(\calO_Z)$.
Indeed we know that there are differential forms in $\frac{H^0(\calO_{\tX}(K+Z))}{H^0(\calO_{\tX}(K))}$, which have got pole of order $Z$ since $h^1(Z', \calL) < h^1(Z, \calL)$ for every cyle $0 \leq Z' < Z$.

This means that $e_{Z_{new}}(u') \neq 0$ since we have blown up the singularity in generic points, so there is a a differential form in $\frac{H^0(\calO_{\tX}(K+Z))}{H^0(\calO_{\tX}(K))}$
which has got a pole on the exceptional divisor $E_{u'}$, it indeed proves $h^1(\calO_{Z_{r}}) < h^1(\calO_Z)$.

Let's denote by $\Omega_D \subset \frac{H^0(\calO_{\tX}(K+Z))}{H^0(\calO_{\tX}(K))}$ the subset of differential forms, which has got no pole on the divisor $D$, more precisely this means 
that $\Omega_D$ contains the differential forms which vanish on the image of the tangent map $T_D(c^{l'}(Z))$, see \cite{NNA1} Theorem 10.1.1.

Since $h^1(Z', \calL) < h^1(Z, \calL)$ for every cyle $0 \leq Z' < Z$ we know that there is a differential form in $\Omega_D$ which has got pole of order $Z$.

We have blown up the vertex $u$ sequentially in generic points, so we know that a generic differential form $w \in \Omega_D$ has got a pole on the exceptional divisor $E_{u'}$, which yields $h^1(Z_{r}, \calL_{new}| Z_r) <  h^1(Z_{new}, \calL_{new})$.

On the other hand representatives of the differential forms in $\Omega_D$ can be viewed as sections in $H^0(\calO_{\tX}(K+Z))$, let's denote a subspace, which represents $\Omega_D$ by $W \subset H^0(\calO_{\tX}(K+Z))$.

If $p$ is the first point, where we blow up the exceptional divisor $E_u$ towards the last exceptional divisor $E_{u'}$, then there is a codimension $1$ subspace of $W$, let's say $W'$, such that the sections in $W'$ vanish at the point $p$.

It means that there is a codimension $1$ subspace of $\Omega_D$, in which the differential forms hasn't got pole on the exceptional divisor $E_{u'}$.

This arguement shows that indeed $h^1(Z_{r}, \calL_{new} | Z_r) \geq h^1(Z_{new}, \calL_{new}) - 1$ and we get that $h^1(Z_{r}, \calL_{new} | Z_r) = h^1(Z_{new}, \calL_{new}) - 1$.

\end{proof}

\section{Outline of the proofs of Theorem\textbf{A} and Theorem\textbf{B}}

In this section we outline the proofs of Theorem\textbf{A} and Theorem\textbf{B} without going deeply into technical details, but showing the main ideas:

\bigskip

In the proof of Theorem\textbf{A} we have a line bundle $\calL_1 \in \pic^{l'}(Z)$ such that $k = h^1(Z, \calL_1) > q = \chi(-l') - \min_{0 \leq l \leq Z} \chi(-l' + l)$ and we wish to prove that there is a line bundle near $\calL_1$ such that its $h^1$ is $k-1$.

We prove the statement by induction on $h^1(\calO_Z)$, where the base case is trivial, so we can assume that $h^1(\calO_Z) = \alpha$ and the statement is true if $h^1(\calO_Z) < \alpha$.

By a standard arguement using Laufer sequences we can easily restrict to the case, when $Z > E$ and $l' \in -S'$.  

Next we prove the statement first in the case when the line bundle $\calL_1$ hasn't got fixed components, or equivalently $H^0(Z, \calL_1)_{reg} \neq \emptyset$, let us have a section $s \in H^0(Z, \calL_1)_{reg}$ and let its divisor be $D = |s| $.

Let the cohomological cyce of the line bundle $\calL_1$ be $Z'$, so $Z'$ is the minimal cycle such that $h^1(Z', \calL_1) = h^1(Z, \calL_1)$.
With another technical step we can easily restrict to the cycle $Z'$ and we are enough to prove the statement for the cycle $Z'$ and the line bundle $\calL_1 | Z'$.

\underline{After these technical steps comes the main part of the proof:}

We find a vertex $u$ such that $Z'_u \geq 2$ and blow up the exceptional divisor $E_u$ in generic points sequentially $Z'_u -1$ times, let us denote the final $-1$ curve by $E_{u'}$,
the new vertex set by $\calv_{new}$ and the new resolution by $\tX_{new}$.

Let us denote furthermore the cycle $Z'_{new} = \pi^*(Z') - \sum_{v \in \calv_{new} \setminus \calv} d(v) E_v$, where $d(v)$ is the distance of $v$ from the vertex set $\calv$, we have $h^1(\calO_{Z'_{new}}) = h^1(\calO_{Z'})$.

Let us have the line bundle $\calL_{new} = \pi^*(\calL_1)$ on the cycle $Z'_{new}$, we have $h^1(Z'_{new}, \calL_{new}) = h^1(Z, \calL_1) = k$ and there is a section $s_{new}\in H^0(Z'_{new},\calL_{new} )_{reg}$, such that $|s_{new}| = D| Z'$.

Let us denote the cycle $Z'_{new} | (\calv_{new} \setminus u') = Z'_{r}$, by our main Lemma\ref{minus} we get easily that $h^1(Z'_{r}, \calL_{new} | Z'_r) = h^1(Z'_{new}, \calL_{new}) - 1$ and $h^1(\calO_{Z'_{r}}) < h^1(\calO_{Z'})$.

Next we want to perturb the line bundle $\calL_{new}$ in $\pic^{\pi^*(l')}(Z'_{new})$ such that its restriction remains $\calL_{new} | Z'_r  := \calL_r$.

Let us have the restriction map $r: \pic^{\pi^*(l')}(Z'_{new}) \to \pic^{R(\pi^*(l'))}(Z'_{r})$, where $R$ is the cohomological restriction operator and notice that 
$\calL_{new} \in r^{-1}(\calL_r)$.

Let us have a generic line bundle $\calL'$ in $r^{-1}(\calL_r)$ near $\calL_{new}$, since $\calL' | Z'_r = \calL_r$ we get that 
$$h^1(Z'_{new}, \calL_{new})    \geq h^1(Z'_{new}, \calL') \geq h^1(Z'_{new}, \calL_{new}) - 1.$$

In the case $h^1(Z'_{new}, \calL') = h^1(Z'_{new}, \calL_{new}) - 1$ we can conclude easily our statement by pushing $\calL'$ into $\pic^{l'}(Z')$ with the isomorphism $\pic^{\pi^*(l')}(Z'_{new}) \cong \pic^{l'}(Z')$.

It means that we can assume that $h^1(Z'_{new}, \calL') = h^1(Z'_{new}, \calL_{new}) = k$, which means that the cohomology number $h^1(Z'_{new}, \calL')$ equals the cohomology of a relatively generic line bundle in $r^{-1}(\calL_r)$.

It means by Theorem\ref{th:hegy2rel} that there exists an effective cycle $0 \leq B \leq Z'_{new}$ such that we have:

\begin{equation*}
k = h^1(Z'_{new}, \calL') = \chi(-\pi^*(l')) - ( \chi( -\pi^*(l') + B) - h^1( (Z'_{new}- B)_r, \calL_r(-B) )).
\end{equation*}

Let us denote $q' =  h^1( (Z'_{new}- B)_r, \calL'' )$, where $\calL''$ is a generic line bundle in $\pic^{R(\pi^*(l') - B)}((Z'_{new}- B)_r)$.

One gets easily that $h^1(\calO_{(Z'_{new}- B)_r}) \leq h^1(\calO_{Z'_r}) < h^1(\calO_{Z'})$ and $h^1( (Z'_{new}- B)_r, \calL_r(-B)) > q'$.

By the induction hypothesis we get that there is a line bundle $\calL'_r$ on the cycle $Z'_{r}$, which is a degeneration of the line bundle $\calL_r$, such that $h^1( (Z'_{new}- B)_r, \calL'_r(-B)) = h^1( (Z'_{new}- B)_r, \calL_r(-B)) - 1$.

Let us have a relatively generic line bundle $\calL_2 \in \pic^{\pi^*(l')}(Z'_{new})$ corresponding to the modified restricted line bundle $\calL'_r$ on the cycle $Z'_{r}$.

Since $\calL_2$ is a degeneration of $\calL'$ we immediately know that $h^1(Z'_{new}, \calL_2) \leq h^1(Z'_{new}, \calL') = k$, on the other hand we get easily again using Theorem\ref{th:hegy2rel} that $h^1(Z'_{new}, \calL_2) \geq k-1$.

If $h^1(Z'_{new}, \calL_2) = k-1$, then we can conlude easily.
If $h^1(Z'_{new}, \calL_2) = k$ then we can repeat the same procedure again to get another degeneration.

One can show that this procedure must terminate in finite time which finishes the proof in the case $H^0(Z, \calL_1)_{reg} \neq \emptyset$.
The case when $H^0(Z, \calL_1)_{reg} = \emptyset$ can be shown easily using the already proved case $H^0(Z, \calL_1)_{reg} \neq \emptyset$ which finishes the proof of Theorem \textbf{A}

\bigskip

The method for proving Theorem\textbf{B} is very similar, only technically more difficult since it depends also on Laufer's deformation theoretic results, so we will give only a short
outline of the proof in philosophical terms.

Let us have a resolution graph $\mathcal{T}$, an effective cycle $Z$ and a Chern class $l'$ which satisfy the conditions of the theorem and suppose that we have a singularity $(X, 0)$
with resolution $\tX$ and resolution graph $\mathcal{T}$ such that $k = h^1(\calO_Z(l')) > \chi(-l') - \min_{0 \leq l \leq Z} \chi(-l' + l) := q$.

We want to find another singularity with resolution $\tX'$ and resolution graph $\mathcal{T}$ such that the value $h^1(\calO_Z(l'))$ on the new resolution equals $k-1$.

By similar technical steps as in the proof of Theorem\textbf{A} we can assume that $|Z|$ is connected and $(l', E_v) \geq 0$ for every $v \in |Z|$ (notice that we cannot assume $Z \geq E$
in this case).

We prove the statement by induction on $h^1(\calO_Z)$, the base case is trivial so assume that $h^1(\calO_Z) = \alpha$ and we know the statement if $h^1(\calO_Z) < \alpha$.

We again prove the statement in the case when the line bundle $\calO_Z(l')$ hasn't got fixed components, or equivalently $H^0(\calO_Z(l'))_{reg} \neq \emptyset$, the general case 
can be proved easily using this special case. 

Let us have a section $s \in H^0(\calO_Z(l'))_{reg}$ and its divisor $D = |s|$.

Let the cohomological cyce of the line bundle $\calO_Z(l')$ be $Z''$, so $Z''$ is the minimal cycle such that $h^1(\calO_{Z''}(l')) = h^1(\calO_Z(l'))$.
With another technical step we can easily restrict to the cycle $Z''$ and we are enough to prove the statement for the cycle $Z''$ and the restsricted natural line bundle $\calO_{Z''}(l')$.

\underline{After these technical steps comes the main part of the proof:}

One can show that there is a vertex $u$ such that $Z''_u \geq 2$ and let us blow up the exceptional divisor $E_u$ in a generic points sequentially $(Z''_u - 1)$ times.

Let us denote the final $-1$-curve by $E_{u'}$, the new vertex set by $\calv_{new}$, and the new resolution by $\tX_{new}$.

Let us denote furthermore  the cycle $Z''_{new} = \pi^*(Z'') - \sum_{v \in \calv_{new} \setminus \calv} d(v) E_v$, where $d(v)$ is the distance of $v$ from the vertex set $\calv$, we have $h^1(\calO_{Z''_{new}}) = h^1(\calO_{Z''})$.

Let's have the line bundle $\calO_{Z''_{new}}(\pi^*(l')) = \pi^*(\calO_{Z''}(l')) | Z''_{new}$ on the cycle $Z''_{new}$, we have $h^1(\calO_{Z''_{new}}(\pi^*(l'))) = h^1(\calO_{Z''}(l')) = k$ and there is a section $s_{new}\in H^0(\calO_{Z''_{new}}(\pi^*(l')))_{reg}$, such that $|s_{new}| = D|Z''$ (since the origins of the blow ups were disjoint from $|D|$).

Let us denote the cycle $Z''_{new} | (\calv_{new} \setminus u') = Z''_{r}$, using our main Lemma\ref{minus} one can show that $h^1(\calO_{Z''_{r}}) < h^1(\calO_{Z''})$ and
$h^1( \calO_{Z''_{r}}(\pi^*(l')) ) = h^1( \calO_{Z''_{new}}(\pi^*(l'))) - 1 = k-1$.

Next what we do is very similar to what we did in the proof of Theorem\textbf{A}, only with some technical modifications what we ommit here for the sake of simpleness:

We deform the analytic structure of the cycle $Z''_{new}$ while keeping the analytic structure of $Z''_r$ fixed and let us denote a relatively generic structure by
$Z''_{new, p}$ where $Z''_r$ is also a subcycle of $Z''_{new, p}$.

We get that $h^1( \calO_{Z''_{new, p}}(\pi^*(l'))) \geq h^1( \calO_{Z''_{r}}(\pi^*(l')) ) = k-1$ (here we need a technical modification in the precise proof to make sure that the line bundle $\calO_{Z''_{r}}(\pi^*(l'))$ stays the same for the deformed analytic structure).
On the other hand $h^1( \calO_{Z''_{new, p}}(\pi^*(l'))) \leq h^1( \calO_{Z''_{new}}(\pi^*(l'))) = k$ since the analytic structure $Z''_{new, p}$ is a degeneration of $Z''_{new}$.

If $h^1( \calO_{Z''_{new, p}}(\pi^*(l'))) = k-1$ then we can get our statement easily so we can assume that for the relatively generic analytic structure $Z''_{new, p}$ one has also
$h^1( \calO_{Z''_{new, p}}(\pi^*(l'))) = k$.

Next from Theorem\ref{relgen2} we compute $h^1( \calO_{Z''_{new, p}}(\pi^*(l')))$ from cohomology numbers of restricted natural line bundles on the cycle $Z''_r$.
We use our induction hypothesis for certain natural line bundles on $Z''_r$ to deform its analytic structure properly such that a corresponding relatively generic analytic structure $Z''_{new, p_2}$ which will satisfy $h^1( \calO_{Z''_{new, p_2}}(\pi^*(l'))) = k-1$.
This will finish the proof in the case $H^0(\calO_Z(l'))_{reg} \neq \emptyset$ and then the general case will follow from it easily, which proves Theorem\textbf{B} completely.

\section{The possible values of $h^1$ of line bundles}

In the case of elliptic singularities if we fix a minimal elliptic resolution graph $\mathcal{T}$, and we look at the possible analytic singularities supported on this resolution graph, then the possible values of $p_g$ can be anything between its minimal value $1$, and it's maximal value, which is the Seiberg-Witten invariant in this case, see for example \cite{R}.

We want to generalise this statement to arbitrary rational homology sphere resolution graphs $\mathcal{T}$, and instead of $p_g = h^1(\calO_{\tX})$ to $h^1(\calO_{Z}(l'))$, so to the values of $h^1$ of an arbitrary natural line bundle on an arbitrary cycle.
 
Although in these cases we don't know the maximal value of these invariants in general, we will still prove that the possible values form an interval.

In this section we consider first a somewhat easier version of our main problem, namely we don't change the analytic types on the resolution graph $\mathcal{T}$, but we fix an analytic type 
$\tX$ and a Chern class $l'$, an effective cycle $Z$ and we look at the $h^1$ stratification of $\pic^{l'}(Z)$, we prove the following theorem in this case:

\begin{theorem}
Let us have an arbitrary rational homology sphere resolution graph $\mathcal{T}$ and a corresponding singularity $(X, 0)$ with resolution $\tX$ and resolution graph $\mathcal{T}$, an effective cycle $Z$ and an abitrary Chern class $l'$. 

Let us denote $k = \max_{\calL \in \pic^{l'}(Z)} h^1(Z, \calL)$ and an arbitrary integer $r$ such that $ \chi(-l') - \min_{0 \leq l \leq Z} \chi(-l'+ l) \leq r \leq k$, then there is a line bundle $\calL \in \pic^{l'}(Z)$, such that $h^1(Z, \calL) = r$.
\end{theorem}

\begin{proof}

If $l' \notin - S'_{|Z|}$, then by Laufer sequence we can find a cycle $l \in L$, such that $R(l'-l) \in - S'_{|Z-l|}$ and $h^0(Z, \calL) = h^0(Z-l, \calL(- l))$
for every line bundle $\calL \in \pic^{l'}(Z)$.

It can happen, that $Z-l$ is nonconnected, but we can treat it's connected components seperately, which means that we can assume in the following that $Z > E$ and
$l' \in - S'$, so we have $l' = \sum_{v \in \calv} a_v E_v^*$ and $a_v \leq 0$ for each $v \in \calv$.

We prove the statement by induction on $h^1(\calO_Z)$.

If $h^1(\calO_Z) = 0$, then $\pic^{l'}(Z)$ consists of just one line bundle, so we have $\min_{\calL \in \pic^{l'}(Z)} h^1(Z, \calL) = \chi(-l') - \min_{0 \leq l \leq Z} \chi(-l'+ l)$, and there is nothing to prove in this case.

So assume in the following that $h^1(\calO_Z) = \alpha$ and we know the statement if $h^1(\calO_Z) < \alpha$.

From \cite{NNA1} we know that for a generic line bundle $\calL' \in \pic^{l'}(Z)$ one has $h^1(Z, \calL') = \chi(-l') - \min_{0 \leq l \leq Z} \chi(-l'+ l) := q$.

It means that $q$ is the minimal possible value of $h^1$ of the line bundles in $\pic^{l'}(Z)$.

So assume in the following that there is a line bundle $\calL_1 \in \pic^{l'}(Z)$ with $h^1(Z, \calL_1) = k$ and assume that $k > q$, obviously we are enough to construct a line bundle near $\calL_1$, whose $h^1$ is $k-1$.

We know that there is a minimal cycle $A$, such that $H^0(Z, \calL_1) = H^0(Z- A, \calL_1(- A)) $ and $\calL_1(- A)  \in \im( c^{l' - A}(Z-A))$, $A$ is the fixed component of the line bundle $\calL_1$ on $Z$, let's prove the statement first in a special case:

\subsection{Proof in the case $A = 0$, or equivalently $\calL_1 \in \im( c^{l' }(Z))$.}

In fact we will prove that there is a map $f: (\bC, 0) \to (\pic^{l'}(Z), \calL_1)$, such that $f(0) = \calL_1$ and for a generic element $t \in (\bC, 0) $, one has $h^1(f(t)) = k-1$.

Notice that the subsets $W_{l'}^r = (\calL \in \pic^{l'}(Z) | h^0(Z, \calL) \geq r)$ are constructible analytic subsets of $\pic^{l'}(Z)$.

It means that our statement is equivalent to the statement that there are line bundles $\calL \in \pic^{l'}(Z)$ arbitrary close to $\calL_1$ in the classical topology, such that $h^1(Z, \calL) = k-1$ (theoretically we will really just use this formulation in the proofs, but in terms of the notations sometimes it will be more convinient to use the formulation with the map $f$).

So assume in the following that $\calL_1 \in \im( c^{l' }(Z))$ and let's have a generic section $s \in H^0(Z, \calL_1)_{reg}$ and its divisor $|s| = D$.

Let's denote by $\Omega_D \subset \frac{H^0(\calO_{\tX}(K+Z))}{H^0(\calO_{\tX}(K))}$ the subset of differential forms, which has got no pole on the divisor $D$, more precisely this means
that $\Omega_D$ contains the differential forms which vanish on the image of the tangent map $T_D(c^{l'}(Z))$.

Assume that a generic differential form $w \in \Omega_D$ has got a pole of order $Z'$ on the exceptional divisors, where $Z'_v = 0$ if $w$ hasn't got a pole on $E_v$, now obviously we have $Z' \leq Z$, we know that $\Omega_D \subset \frac{H^0(\calO_{\tX}(K+Z'))}{H^0(\calO_{\tX}(K))}$.

\begin{lemma}
We can reduce our statement to $ Z'$, so we are enough to prove that there is a map $f': (\bC, 0) \to \pic^{l'}(Z')$, such that $f'(0) = \calL_1 | Z'$ and for a generic element $t \in (\bC, 0) $, one has $h^1(f'(t)) = k-1$.
\end{lemma}
\begin{proof}

Notice that $\Omega_D$ contains differential forms which vanish on the image of the tangent map $T_{\tau(D)}(c^{l'}(Z'))$, where $\tau$ is the restriction $\tau: \eca^{l'}(Z) \to \eca^{l'}(Z')$.

This means that $h^1(Z', \calL_1|Z') \geq h^1(Z, \calL_1)$ but $h^1(Z', \calL_1|Z') \leq h^1(Z, \calL_1)$ holds obviously so we have $h^1(Z', \calL_1 | Z') = h^1(Z, \calL_1) = k$.

We know that $Z'$ is the minimal cycle, such that $Z' \leq Z$ and $h^1(Z', \calL_1|Z') = h^1(Z, \calL_1) = k$, and obviously we have $h^1(Z', \calL_1) >  \chi(-l') - \min_{0 \leq l \leq Z'} \chi(-l'+ l)$.

If we know the statement in case $ Z'$ then we get that there is a map $f': (\bC, 0) \to \pic^{l'}(Z')$, such that $f(0) = \calL_1 | Z'$ and for a generic element $t \in (\bC, 0) $, one has $h^1(f'(t)) = k-1$.

We can extend the map $f'$ to a map $f'' : (\bC, 0) \to \pic^{l'}(Z)$, such that $f' = \pi \circ f''$, where $\pi : \pic^{l'}(Z) \to \pic^{l'}(Z')$ is the natural projection.

Now for a generic element $t \in (\bC, 0) $ we have $h^1(f'(t)) = k-1$ and $ h^1(f''(t)) \geq  h^1(f'(t)) = k-1$ and by semicontinuity we have $ h^1(f''(t)) \leq h^1(Z, \calL_1) = k$.

If $h^1(f''(t)) = k-1$, then we are done, so we can assume that $k = h^1(f''(t)) > h^1(f'(t))$, which means that $Z'$ isn't the minimal cycle, such that $Z' \leq Z$ and $h^1(Z', f''(t)) = h^1(Z, f''(t))$.

We know that $f''(t)$ is a degeneration of $\calL_1$ with $k = h^1(f''(t))$ and this minimal cycle can change only finitely many times by passing to degenerations.

Indeed the minimal cycle $Z'$, such that $Z' \leq Z$ and $h^1(Z', \calL_1) = h^1(Z, \calL_1)$ depends on the set of cohomology numbers $(h^1(A, \calL_1) | A \leq Z)$, so if $Z'$
decreases after a degeneration of the line bundle $\calL_1$ then one of the cohmology numbers $h^1(A, \calL_1)$ decreases which can happen only finitely many times since
$\sum_{0 \leq A \leq Z} h^1(A, \calL_1)$ is finite.

It means that repeating this argument finitely many times we get a degeneration $\calL$ of the line bundle $\calL_1$, such that $h^1(Z, \calL) = k-1$ and this proves our lemma.
\end{proof}

So we are enough to prove that there is a map $f': (\bC, 0) \to \pic^{l'}(Z')$, such that $f'(0) = \calL_1 | Z'$ and for a generic element $t \in (\bC, 0) $, one has $h^1(f'(t)) = k-1$.

It is obviuous as usual that we can assume $Z' \geq E$ and so $|Z'|$ is connencted, and if $h^1(\calO_{Z'}) < h^1(\calO_Z)$, then the statement follows from the induction hypothesis, so we can assume that $h^1(\calO_{Z'}) = h^1(\calO_Z) > 0$.

It means that we have $\Omega_D \subset \frac{H^0(\calO_{\tX}(K+Z'))}{H^0(\calO_{\tX}(K))} = \frac{H^0(\calO_{\tX}(K+Z))}{H^0(\calO_{\tX}(K))}$.

Let's denote the subset of vertices $I = (v \in \calv | Z'_v \geq 2)$, we know that $I \neq 0$ because of $h^1(\calO_{Z'}) > 0$, so let's pick an arbitrary vertex $u \in I$.

Let's blow up the exceptional divisors $E_u$ in generic points sequentially $Z'_u - 1$ times, let's denote the final $-1$ curve by $E_{u'}$, let's denote the new vertex set by $\calv_{new}$, the new resolution by $\tX_{new}$.

Let's denote furthermore $Z'_{new} = \pi^*(Z') - \sum_{v \in \calv_{new} \setminus \calv} d(v) E_v$, where $d(v)$ is the distance of $v$ from the vertex set $\calv$, we have $h^1(\calO_{Z'_{new}}) = h^1(\calO_{Z'})$.

Notice that this means that every differential form in $\frac{H^0(\calO_{\tX}(K+Z'))}{H^0(\calO_{\tX}(K))} $ has got a pole on the exceptional divisor $E_{u'}$ of order at most $1$.

Let's have the line bundle $\calL_{new} = \pi^*(\calL_1)$ on the cycle $Z'_{new}$, we have $h^1(Z'_{new}, \calL_{new}) = h^1(Z, \calL_1) = k$ and there is a section $s_{new}\in H^0(Z'_{new},\calL_{new} )_{reg}$, such that $|s_{new}| = D$ (since the origins of the blow ups are disjoint from $|D|$).

Let's denote the cycle $Z'_{new} | (\calv_{new} \setminus u') = Z'_{r}$, we claim that $h^1(\calO_{Z'_{r}}) < h^1(\calO_{Z'})$:

We know that there are differential forms in $\frac{H^0(\calO_{\tX}(K+Z'))}{H^0(\calO_{\tX}(K))} $, which have got pole of order $Z'$.

Indeed $Z'$ is the minimal cycle such that $h^1(Z', \calL_1) = h^1(Z, \calL_1)$, so the generic element of $\Omega_D$ has got pole of order $Z'$.

It means that $e_{Z'_{new}}(u') \neq 0$, which indeed yields $h^1(\calO_{Z'_{r}}) < h^1(\calO_{Z'})$.

Now by Lemma\ref{minus} we have that $h^1(Z'_{r}, \calL_{new} | Z'_r) = h^1(Z'_{new}, \calL_{new}) - 1$.

\textbf{Next we want to perturb the line bundle $\calL_{new}$ in $\pic^{\pi^*(l')}(Z'_{new})$ such that its restriction remains $\calL_{new} | Z'_r$:}

Notice that the candidate line bundles are in $\calL_{new} + V_{Z'_{new}}(u') \subset \pic^{\pi^*(l')}(Z'_{new}) \cong \pic^{l'}(Z)$. 
Since we want the $h^1$ to be the least as possible, we want to chose most generic line bundle in $\calL_{new} + V_{Z'_{new}}(u')$.

So let's have a generic line bundle $\calL'$ in $ \calL_{new} + V_{Z'_{new}}(u') \subset \pic^{\pi^*(l')}(Z'_{new}) \cong \pic^{l'}(Z)$ near the line bundle $\calL_{new}$.

We have by semicontinuity of $h^1$ that $h^1(Z'_{new}, \calL') \leq h^1(Z'_{new}, \calL_{new}) = k$ and on the other hand we have $\calL' | Z'_r = \calL_{new}|  Z'_r$, which yields that $h^1(Z'_{new}, \calL') \geq h^1(Z'_{r}, \calL_{new} | Z'_r) = h^1(Z'_{new}, \calL_{new}) - 1 = k-1$.

\underline{Assume first that $h^1(Z'_{new}, \calL') = k-1$:}

In this case by the isomorphism $ \pic^{\pi^*(l')}(Z'_{new}) \cong \pic^{l'}(Z')$, which is given by the pullback of the blowup map one gets a line bundle $\calL'_1$ arbitrary near to $\calL_1$, such that $h^1(Z', \calL'_1) = k-1$, and we are done.

\underline{It means that we can assume $h^1(Z'_{new}, \calL') = k$:}

Notice that $\calL'$ is a relatively generic line bundle on $Z'_{new}$ with Chern class $\pi^*(l')$ corresponding to the restricted line bundle $\calL' | Z'_r = \calL_{new}|  Z'_r$, let's denote this restricted line bundle by $\calL_r$.

By Theorem\ref{th:hegy2rel} we get that:

\begin{equation*}
k = h^1(Z'_{new}, \calL') = \chi(-\pi^*(l')) - \min_{0 \leq l \leq Z'_{new}}( \chi( -\pi^*(l') + l) - h^1( (Z'_{new}- l)_r, \calL_r(-l) )).
\end{equation*}

It means that there exists an effective cycle $0 \leq B \leq Z'_{new}$ such that we have:

\begin{equation*}
k = h^1(Z'_{new}, \calL') = \chi(-\pi^*(l')) - ( \chi( -\pi^*(l') + B) - h^1( (Z'_{new}- B)_r, \calL_r(-B) ) ).
\end{equation*}

It is obvious that $h^1(\calO_{(Z'_{new}- B)_r}) \leq h^1(\calO_{Z'_r}) < h^1(\calO_{Z'})$ and let's denote $q' =  h^1( (Z'_{new}- B)_r, \calL'' )$, where $\calL''$ is a generic line bundle in $\pic^{R(\pi^*(l') - B)}((Z'_{new}- B)_r)$.

If $\calL_{gen} \in \pic^{\pi^*(l')}(Z'_{new})$ is a generic line bundle, then we have the exact sequence:

\begin{equation*}
0 \to    H^0(Z'_{new}- B, \calL_{gen}(- B))    \to H^0(Z'_{new}, \calL_{gen}) \to H^0(B, \calL_{gen}) \to H^1(Z'_{new}- B, \calL_{gen}(- B)),
\end{equation*}

which yields that  $q \geq \chi(-\pi^*(l')) - \chi( -\pi^*(l') + B) + q'$, s \textbf{ it means that we have $h^1( (Z'_{new}- B)_r, \calL_r(-B)) > q'$.}

By the induction hypothesis we know that there is a line bundle $\calL'_r$ on the cycle $Z'_{r}$, which is a degeneration of the line bundle $\calL_r$, such that $h^1( (Z'_{new}- B)_r, \calL'_r(-B)) = h^1( (Z'_{new}- B)_r, \calL_r(-B)) - 1$.

Now let's look at a relatively generic line bundle $\calL_2 \in \pic^{\pi^*(l')}(Z'_{new})$ corresponding to the modified restricted line bundle $\calL'_r$ on the cycle $Z'_{r}$.

Since $\calL_2$ is a degeneration of $\calL'$ we immediately know that $h^1(Z'_{new}, \calL_2) \leq h^1(Z'_{new}, \calL') = k$.

On the other hand we have:
\begin{equation*}
h^1(Z'_{new}, \calL_2) \geq  \chi(-\pi^*(l')) -  \chi( -\pi^*(l') + B) +  h^1( (Z'_{new}- B)_r, \calL_2(-B)| (Z'_{new}- B)_r ).
\end{equation*}

\begin{equation*}
h^1(Z'_{new}, \calL_2) \geq \chi(-\pi^*(l')) -  \chi( -\pi^*(l') + B) +  h^1( (Z'_{new}- B)_r,  \calL'_r(-B)).
\end{equation*}

Since we know that $h^1( (Z'_{new}- B)_r, \calL'_r(-B)) = h^1( (Z'_{new}- B)_r, \calL_r(-B)) - 1$ we get that $h^1(Z'_{new}, \calL_2) \geq h^1(Z'_{new}, \calL') - 1 = k-1$.

If $h^1(Z'_{new}, \calL_2) = k-1$, then we are done, on the other hand if $h^1(Z'_{new}, \calL_2) = k$ we can repeat the same procedure again to get another degeneration.

Notice that while this degeneration step we have:
\begin{equation*}
\sum_{0 \leq l \leq Z'_{new}} h^1( (Z'_{new}- l)_r, \calL'(-l) | (Z'_{new}- l)_r) < \sum_{0 \leq l \leq Z'_{new}} h^1( (Z'_{new}- l)_r, \calL_2(-l) | (Z'_{new}- l)_r).
\end{equation*}

Indeed, this is because of semicontinuity of $h^1$ and because of the fact that $h^1( (Z'_{new}- B)_r, \calL'_r(-B)) = h^1( (Z'_{new}- B)_r, \calL_r(-B)) - 1$.

This means that this degeneration process can repeat only finitely many times, so finally we got a degeneration $\calL_1$ of the line bundle $\calL_{new}$ such that $h^1(Z'_{new}, \calL_1) = k-1$.

Using the isomorphism $ \pic^{\pi^*(l')}(Z'_{new}) \cong \pic^{l'}(Z)$ this completes the proof in the case $A = 0$.

\subsection{Proof in the general case}
Next we want to prove the statement in the general case without the assumption $A = 0$.

We know that there are only finitely many possibilities for the fixed component cycle $A$ of a line bundle $\calL_1 \in  \pic^{l'}(Z)$, namely the cycles $0 \leq A \leq Z$.

Assume that $\calL_1 \in \pic^{l'}(Z)$ with $h^1(Z, \calL_1) = k$ and $k > q$, and let's denote the fixed component cycle of the line bundle $\calL_1 \in \pic^{l'}(Z)$ by $A$.

It means that we have:
\begin{equation*}
k = h^1(Z, \calL_1) = \chi(-l') - \chi(-l' + A) + h^1(Z-A, \calL_1(- A)).
\end{equation*}

We know obviously that:
\begin{equation*}
\chi(-l') - \chi(-l' + A) + \chi(-l' + A) - \min_{0 \leq l \leq Z-A} \chi(-l' + A + l) \leq \chi(-l') - \min_{0 \leq l \leq Z} \chi(-l'+ l),
\end{equation*}
which means that $h^1(Z-A, \calL_1(- A)) > \chi(-l' + A) - \min_{0 \leq l \leq Z-A} \chi(-l' + A + l)$.

\emph{We know that $\calL_1(- A) \in \im( c^{l' - A}(Z-A))$, so we can apply the statement already proved in the special case $A = 0$:}

This means that there is a map $f: (\bC, 0) \to \pic^{l' - A}(Z-A)$, such that $f(0) = \calL_1(- A) $ and for a generic element $t \in (\bC, 0) $, one has $h^1(f(t)) =  h^1(Z-A, \calL_1(- A)) - 1$., let's denote $f(u)$ by $\calL_u$.

Let's have now a generic element $t \in (\bC, 0) $ and let's denote $f(t) = \calL_t$, so we have $h^1(Z-A, \calL_t) = h^1(Z-A, \calL_1(- A)) - 1$.

Now let's have a map $g: (\bC, 0) \to \pic^{l'}(Z)$, such that $g(u)(-A)  | (Z-A) = f(u)$.

For a generic element $u \in (\bC, 0)$ let's denote the line bundle $g(u)$ by $\calL_{u, 1} \in \pic^{l'}(Z)$, where we have $\calL_{u, 1}(- A) | Z-A = \calL_u$.

We have that $h^1(Z, \calL_{u, 1}) \geq \chi(-l') - \chi(-l' + A) + h^1(Z-A, \calL_u) = k-1$.

On the other hand  by semicontinuity of $h^1$ we know that $h^1(Z, \calL_{u, 1}) \leq h^1(Z, \calL_{ 1}) = k$.

If $h^1(Z, \calL_{u, 1}) = k-1$, then we are done, so we can assume that for a generic element $u\in (\bC, 0)$ one has $h^1(Z, \calL_{u, 1}) = k$.

This means that the fixed point cycle of the line bundle $\calL_{u, 1}$ on the cycle $Z$ is not $A$.

Now notice that the fixed point cycle $A$ is determined by the set of cohomology numbers $(h^0(Z- l - E_I, \calL_1(- l - E_I)) | 0 \leq l \leq Z, I \subset |Z-l|)$.

Indeed $A$ is the minimal cycle such that $h^0(Z- A - E_I, \calL_1(- A - E_I)) < h^0(Z- A , \calL_1(- A))$ for every set $ I \subset |Z-A|$.

Notice that the line bundle $\calL_{u, 1}$ is a degeneration of the line bundle $\calL_1$, which means that $\sum_{0 \leq l \leq Z, I \subset |Z-l|} h^0(Z- l - E_I, \calL_1(- l - E_I)) > \sum_{0 \leq l \leq Z, I \subset |Z-l|} h^0(Z- l - E_I, \calL_{u, 1} - l - E_I)$, because of semicontinuity of $h^0$, and because of the fact that the fixed point cycle changes.

This means that we can repeat this degeneration process only finitely many times, and finally we get a degeneration $\calL$ of the line bundle $\calL_1$, such that $h^1(Z, \calL) = k-1$.
This proves our theorem in the general case.
\end{proof}

\section{The possible values of $h^1$ of natural line bundles}

In the following we want to prove the same type of theorem as in the last section, but now we are looking at the values of the $h^1$ cohomology number of natural line bundles like $\calO_Z(l')$ while we are moving the analytic structure of the singularity fixing the topological type.

The proof of our main theorem will be similar as in the case of line bundles, but it will be more technical since we should Laufer's deformation results in it:

\begin{theorem}
Let $\mathcal{T}$ be an arbitrary rational homology sphere resolution graph with vertex set $\calv$ and let $Z$ be an effective cycle on it, and suppose furthermore that $l' \in L'$ is a Chern class, such that $l' = \sum_{v \in \calv} a_v E_v^* \in L'$. Let us write $l' =  \sum_{v \in \calv} b_v E_v$ where $b_v$ are rational numbers, and assume that $b_v <0$ if $v \in |Z|$.

Suppose that we have a singularity $(X, 0)$ with resolution $\tX$ and resolution graph $\mathcal{T}$, and let us have the restricted natural line bundle $\calO_Z(l')$.

Suppose that $k = h^1(\calO_Z(l')) > \chi(-l') - \min_{0 \leq l \leq Z}  \chi(-l' + l)$, and let us have an arbitrary number $k > r \geq \chi(-l') - \min_{0 \leq l \leq Z}  \chi(-l' + l)$, then there is another singularity with resolution $\tX'$ and with resolution graph $\mathcal{T}$, for which one has $r = h^1(\calO_Z(l'))$.
\end{theorem}

\begin{proof}

Let's denote the restricton of the Chern class $l'$ to $|Z|$ by $l'_Z \in L'_{|Z|}$.

Assume now that  $l'_Z \notin - S'_{|Z|}$, then by Laufer sequence we can find a cycle $l \in L$, such that $R(l'-l) \in - S'_{|Z-l|}$ and $h^0(\calO_Z(l')) = h^0(\calO_{Z-l}(l'-l))$
for every singularity corresponding to the resolution graph $\mathcal{T}$.

Notice that the negativity assumption of the coefficients still hold for the natural line bundle $\calO_{Z-l}(l'-l)$.

It can happen that $|Z-l|$ is nonconnected, but we can treat its connected components seperately.

\emph{It means that we can assume in the following that $|Z|$ is connected and $l'_Z \in - S'_{|Z|}$, so we have $a_v < 0$ for each vertex $v \in |Z|$:}

Now from \cite{NNA1} we know that for a generic line bundle $\calL' \in \pic^{l'_Z}(Z)$ one has:

\begin{equation*}
h^1(Z, \calL') = \chi(-l') - \min_{0 \leq l \leq Z} \chi(-l'+ l) = \chi(-l'_Z) - \min_{0 \leq l \leq Z} \chi(-l'_Z+ l),
\end{equation*}
let's denote this number by $q$ for the rest of the proof.

It means that $q$ is the minimal possible value of $h^1$ of the line bundles in $\pic^{l'_Z}(Z)$.

We will prove the theorem by induction on the parameter $h^1(\calO_Z)$.

If $h^1(\calO_Z) = 0$, then $\pic^{l'}(Z)$ consists of just one line bundle and then $h^1(\calO_Z(l')) = \chi(-l') - \min_{0 \leq l \leq Z}  \chi(-l' + l)$ and there is nothing to prove.

So assume in the following that $h^1(\calO_Z) = \alpha$ and we know the statement in case $h^1(\calO_Z) < \alpha$.

Notice that from \cite{NNA2} Theorem 5.1.1 , we know that for a generic singularity $\tX$ supported on the graph $\mathcal{T}$ we have:

\begin{equation*}
h^1(\calO_Z(l')) = \chi(-l') - \min_{0 \leq l \leq Z}\chi(-l' + l) = q.
\end{equation*}

This means that $q$ is the minimal possible value $h^1(\calO_Z(l'))$ can be, for any analytic structure.

So let's return to the proof of our main theorem and assume, that $k > q$, and we have a singularity $\tX$ supported on $\mathcal{T}$, such that $k = h^1(\calO_Z(l')) > q$.
We are enough  to construct a singularity $\tX'$ supported on the same resolution graph, such that $h^1(\calO_Z(l')) = k-1$.

Let's have a very large cycle $Z' > Z$ on the singularity $\tX$ and let's apply now Laufer's result, which says that there is a deformation $\omegl:\calz' \to (Q, 0)$ of $Z'$, such that $\omegl^{-1}(0) = Z'$, which is complete at every point $q \in Q$.

We know that there is a minimal cycle $A$, such that $H^0(\calO_{Z}(l')) = H^0(\calO_{Z-A}(l'- A)) $ and $\calO_{Z-A}(l'- A) \in \im( c^{l' - A}(Z-A))$, where $A$ is the fixed component of the line bundle $\calO_Z(l')$ on the cycle $Z$.

\subsection{Proof in the case $A = 0$, or equivalently $\calO_{Z}(l') \in\im( c^{l' }(Z))$.}

We prove the statement first in the special case $A = 0$, in fact we will prove that there is a map $f: (\bC, 0) \to (Q, 0) $, such that $f(0) = 0$ and for a generic element $t \in (\bC, 0) $, if we denote $\omegl^{-1}(f(t)) = Z'_t$, and it's subcycle $Z_t$, then one has $h^1(\calO_{Z_t}(l')) = k-1$.

So assume in the following that $\calO_{Z}(l') \in \im( c^{l' }(Z))$ and let's have a generic section $s \in H^0(\calO_{Z}(l'))_{reg}$ and its divisor $|s| = D$.

Let's denote by $\Omega_D \subset \frac{H^0(\calO_{\tX}(K+Z))}{H^0(\calO_{\tX}( K))}$ the subset of differential forms, which has got no pole on the divisor $D$, more precisely this means
that $\Omega_D$ contains the differential forms which vanish on the image of the tangent map $T_D(c^{l'}(Z))$.

Let's assume that a generic differential form $w \in \Omega_D$ has got a pole of order $Z''$ on the exceptional divisors, where $Z''_v = 0$ if $w$ han't got a pole on $E_v$, now obviously we have $Z'' \leq Z$.

As before we know that $Z''$ is the minimal cycle, such that $Z'' \leq Z$ and $h^1(\calO_{Z''}(l')) = h^1(\calO_{Z}(l')) = k$, and obviously we have $h^1(\calO_{Z''}(l')) \geq  \chi(-l') - \min_{0 \leq l \leq Z''} \chi(-l'+ l)$.

\begin{lemma}
We can reduce the statement to $Z''$ so we are enough to prove that there is a map $f': (\bC, 0) \to (Q, 0)$, such that $f(0) = 0$ and for a generic element $t \in (\bC, 0) $, if we denote $\omegl^{-1}(f'(t)) = Z'_t$, and it's subcycle by $Z''_t$, then one has $h^1(\calO_{Z''_t}(l')) = k-1$.
\end{lemma}
\begin{proof}

Since the subsets $W_r(Q) = (p \in Q | h^1(\calO_{Z_p}(l')) = r)$ are constructible analytic subsets, our statement is equivalent to the statement that there are points $p \in Q$ arbitrary close to $0$ in the classical topology, such that $h^1(\calO_{Z_p}(l')) = k-1$.

For a generic element $t \in (\bC, 0) $ we have $h^1(\calO_{Z''_t}(l')) = k-1$ and $ h^1(\calO_{Z_t}(l')) \geq  h^1(\calO_{Z''_t}(l')) = k-1$ and by semicontinuity of $h^1$ we have $ h^1(\calO_{Z_t}(l')) \leq h^1(\calO_{Z}(l')) = k$.

If $h^1(\calO_{Z_t}(l')) = k-1$, then we are done, so we can assume that $k = h^1(\calO_{Z_t}(l')) > h^1(\calO_{Z''_t}(l'))$, which means that  $Z''_t$ isn't the minimal cycle, such that $Z''_t \leq Z_t$ and $ h^1(\calO_{Z''_t}(l'))= h^1(\calO_{Z_t}(l'))$.

We know that $f'(t)$ is a degeneration of $Z'$ with $k = h^1(\calO_{Z_t}(l'))$ and we want to argue that this minimal cycle can change only finitely many times by passing to degenerations, so repeating this argument a finitely times we get a degeneration $Z'_p$ of $Z'$,such that $h^1(\calO_{Z_p}(l')) = k-1$ and this would prove the statement.

Indeed we know that the minimal cycle $C$, for which $h^1( \calO_{C}(l')) = h^1(\calO_{Z}(l'))$ is determined by the set of numbers $(h^1(\calO_{B}(l')) | B \leq Z)$.

Notice that while passing to degenerations, these numbers can only decrease, which means that if any of them changes, then the sum $\sum_{0 \leq B \leq Z} h^1(\calO_{B}(l'))$ decreases.

It means that indeed the minimal cycle $C$ can change only finitely many times while passing to degenerations.

\end{proof}

So we are enough to prove in the following that there is a map $f': (\bC, 0) \to (Q, 0)$, such that $f'(0) = 0$ and for a generic element $t \in (\bC, 0) $, one has $h^1(\calO_{Z''_t}(l')) = k-1$.

 We can assume that $h^1(\calO_{Z''}) = h^1(\calO_Z) = \alpha$, otherwise the statement follows from the induction hypothesis, notice that we have $\Omega_D \subset \frac{H^0(\calO_{\tX}(K+Z''))}{H^0(\calO_{\tX}(K))} = \frac{H^0(\calO_{\tX}( K+Z))}{H^0(\calO_{\tX}(K))}$.

Let's denote the subset of vertices $I = (v \in \calv | Z''_v \geq 2)$, we know that $I \neq 0$ because of $h^1(\calO_{Z''}) > 0$, so let's pick an arbitrary vertex $u \in I$.

Let's blow up the exceptional divisor $E_u$ in a generic points sequentially $(Z''_u - 1)$ times.

Now let's denote the final $-1$-curve by $E_{u'}$, let's denote the new vertex set by $\calv_{new}$, and the new resolution by $\tX_{new}$.

Let's denote furthermore $Z''_{new} = \pi^*(Z'') - \sum_{v \in \calv_{new} \setminus \calv} d(v) E_v$, where $d(v)$ is the distance of $v$ from the vertex set $\calv$, we obviously have $h^1(\calO_{Z''_{new}}) = h^1(\calO_{Z''})$.

Notice that every differential form in $\frac{H^0(\calO_{\tX}(K+Z''))}{H^0(\calO_{\tX}(K))}$ has got a pole on the exceptional divisor $E_{u'}$ of order at most $1$.

Let's have the line bundle $\calO_{Z''_{new}}(\pi^*(l')) = \pi^*(\calO_{Z''}(l')) | Z''_{new}$ on the cycle $Z''_{new}$, we have $h^1(\calO_{Z''_{new}}(\pi^*(l'))) = h^1(\calO_{Z''}(l')) = k$ and there is a section $s_{new}\in H^0(\calO_{Z''_{new}}(\pi^*(l')))_{reg}$, such that $|s_{new}| = D$ (since the origins of the blow ups were disjoint from $|D|$).

Let's denote the cycle $Z''_{new} | (\calv_{new} \setminus u') = Z''_{r}$, we claim that $h^1(\calO_{Z''_{r}}) < h^1(\calO_{Z''})$:

We know that there are differential forms in $\frac{H^0(\calO_{\tX}(K+Z''))}{H^0(\calO_{\tX}(K))}$, which have got pole of order $Z''$, indeed the generic differentiial form in $\Omega_D$
has got a pole of order $Z''$.

\emph{It means that $e_{Z''_{new}}(u') \neq 0$, which indeed yields $h^1(\calO_{Z''_{r}}) < h^1(\calO_{Z''})$.}

Now by  Lemma\ref{minus} we have $h^1( \calO_{Z''_{r}}(\pi^*(l')) ) = h^1( \calO_{Z''_{new}}(\pi^*(l'))) - 1 = k-1$.

Let's denote $Z'_{new} = \pi^*(Z') - \sum_{v \in \calv_{new} \setminus \calv} d(v) E_v$ and $Z'_{new} | (\calv_{new} \setminus u') = Z'_{r}$.

\textbf{Next we will deform the analytic structure of $Z'_{new}$:}

Let's have a large number $M$, and let's denote the neighbour of the vertex $u'$ by $n(u)$, and let's blow up the divisor $E_{n(u)}$ at the intersection point of $E_{n(u)}$ and $E_u$ sequentially $M$ times (we always blow up at the intersection point of $E_u$ with its neighbour divisor) and let the new divisors be $E_{u_1},E_{ u_2}, \cdots E_{u_M}$.

Let the new singularity be $\tX_{new, M}$ with vertex set $\calv_{new, M}$, where we have $u' \in \calv_{new} \subset \calv_{new, M}$.

Let's denote furthermore $Z''_{new, M} = \pi^*(Z''_{new})$, we have $h^1(\calO_{Z''_{new, M}}) = h^1(\calO_{Z''_{new}})$.

We have still the fact that every differential form in $\frac{H^0(\calO_{\tX}(K+Z''))}{H^0(\calO_{\tX}(K))}$ has got a pole on the exceptional divisor $E_{u'}$ of order at most $1$.

Let's have the line bundle $\calO_{Z''_{new, M}}(\pi^*(l')) = \pi^*(\calO_{Z''_{new}}(\pi^*(l')))$, we obviously have $h^1(\calO_{Z''_{new, M}}(\pi^*(l'))) = h^1(\calO_{Z''}(l')) = k$ and there is a section $s_{new, M}\in H^0(\calO_{Z''_{new, M}}(\pi^*(l')))_{reg}$, such that $|s_{new, M}| = D$.

Let's denote the cycle $Z''_{new, M} | (\calv_{new, M} \setminus u') = Z''_{r, M}$,we know that:

\begin{equation*}
h^1(\calO_{Z''_{r, M}}) = h^1(\calO_{Z''_{r, M}}) < h^1(\calO_{Z''_{new, M}})= h^1(\calO_{Z''}).
\end{equation*}

We know that if $M$ is enough large, then $e_{Z''_{r, M}}(u_M) = 0$ for each each analytic type, see \cite{NNA1}.

We know that $h^1( \calO_{Z''_{r, M}}(\pi^*(l')) ) = h^1( \calO_{Z''_{new, M}}(\pi^*(l'))) - 1 = k-1$, indeed it immediately follows from $\calO_{Z''_{r, M}}(\pi^*(l'))= \pi^*(\calO_{Z''_{r}}(\pi^*(l')))$.

Let's denote $Z'_{new, M} = \pi^*(Z'_{new})$ and $Z'_{new, M} | \calv_{new, M} \setminus u' = Z'_{r, M}$.

\emph{Now by Theorem\ref{th:La2} we have a complete deformation space of the pair $(Z'_{new, M}, Z'_{r, M})$:}

It means, that there is a complete deformation of the pair $(Z'_{new, M}, Z'_{r, M})$, $\omegl_{new}: \calz_{new} \to (Q_u, 0)$, such that $\omegl_{new}$ is a trivial deformation of $Z'_{r, M}$ that is, there is a closed subspace $\calz_r \subset \calz_{new}$ and the restriction of $\omegl_{new}$ to $\calz_r$ is a trivial deformation of $Z'_{r, M}$.

Notice that by blowing down the deformation $\omegl_{new}$ gives a deformation $\omegl_1$ of the cycle $Z'$ with base space $(Q_u, 0)$.

We know that the deformation $\omegl:\calz' \to (Q, 0)$ of $Z'$ is complete, which means that there is a holomorphic map $g: (Q_u, 0) \to (Q, 0)$, such that $g(0) = 0$ and $g$ induces the deformation $\omegl_1$ of the cycle $Z'$.

If there is a point $p \in Q_u$ arbitrary close to $0$, such that $h^1( \calO_{Z''_{new, M, p}}(\pi^*(l'))) = k-1$, then we are done, because then the points $g(p) \in Q$ are
arbitrary close to $0$ and we have $h^1(\calO_{Z''_{g(p)}}(l')) = k-1$.

Let's have a generic point $p \in Q_u$ near $0$, we have by semicontinuity of $h^1$ that $h^1(\calO_{Z''_{new, M, p}}(\pi^*(l'))) \leq h^1( \calO_{Z''_{new, M}}(\pi^*(l'))) = k$.

On the other hand we have $Z''_{r, M, p} \cong  Z''_{r, M}$ and the line bundles $\calO_{Z''_{new, M}}(\pi^*(l')) | Z''_{r, M}$ and $\calO_{Z''_{new, M, p}}(\pi^*(l')) | Z''_{r, M}$ are the same.

Indeed $e_{Z''_{r, M}}( u_M) = 0$ for each analytical type, so the divisors $E_{u'}$ and $E_{u', p}$ give the same line bundle on $ Z''_{r, M}$.

It yields immediately that $h^1( \calO_{Z''_{new, M, p}}(\pi^*(l'))) \geq h^1( \calO_{Z''_{r, M}}(\pi^*(l'))) = k-1$, so we got that $k-1 \leq h^1(  \calO_{Z''_{new, M, p}}(\pi^*(l'))) \leq k$.

If $h^1(  \calO_{Z''_{new, M, p}}(\pi^*(l'))) = k-1$, then we are immediately done.

\bigskip

\underline{Assume in the following that $h^1(\calO_{Z''_{new, M, p}}(\pi^*(l')))= k$:}

\bigskip

We know that $g(p)$ can be arbitrary close to $0$ in $Q$, and by Theorem\ref{relgen2} we get that:

\begin{equation*}
k = h^1( \calO_{Z''_{new, M, p}}(\pi^*(l'))) = \chi(- \pi^*(l')) - \min_{0 \leq l \leq Z''_{new, M, p}}( \chi(- \pi^*(l') + l)  - h^1(\calO_{(Z''_{new, M, p}- l)_{r, M}}(  \pi^*(l') - l))).
\end{equation*}

Let's have a complete deformation $\omegl_2 :\calz'_{new, M} \to (Q_{2}, 0)$ of the cycle $Z'_{new, M, p}$, by blowing down this gives a deformation $\omegl_s$ of the cycle $Z'_{g(p)}$ with base space $(Q_2, 0)$.

Now we know, that the deformation $\omegl:\calz' \to (Q, 0)$ of $Z'$ is complete at $g(p)$, which means that there is a holomorphic map $g': (Q_2, 0) \to (Q, g(p))$, such that $g'(0) = g(p)$ and $g'$ induces the deformation $\omegl_s$ of the cycle $Z'_{g(p)}$.

It means that we are enough to show that there are points $p_2 \in Q_2$ arbitrary close to $0$ such that $k-1 = h^1( \calO_{Z''_{new, M, p_2}}(\pi^*(l'))) = h^1( \calO_{Z''_{g'( p_2)}}(l'))$ since then $g'(p_2)$ can be arbitrary close to $g(p)$.

For an arbitrary $\beta \in Q_2$ let's denote
\begin{equation*}
G(\beta) = \chi(- \pi^*(l')) - \min_{0 \leq l \leq Z''_{new, M, \beta}}( \chi(- \pi^*(l') + l)  - h^1(\calO_{(Z''_{new, M, \beta}- l)_{r, M}}(\pi^*(l') - l))).
\end{equation*}

\begin{lemma}
There are points $p_2 \in Q_2$ arbitrary close to $0$, such that $G(p_2) = k-1$.
\end{lemma}
\begin{proof}

We know that there is a cycle $0 \leq B \leq Z''_{new, M, p}$, such that $k =  \chi(- \pi^*(l')) - \chi(- \pi^*(l') + B) + h^1(\calO_{(Z''_{new, M, p}- B)_{r, M}}(  \pi^*(l') - B))$.

Let's denote by $q' =\chi(-\pi^*(l') + B) -  \min_{0 \leq l \leq (Z''_{new, M, p}- B)_{r, M}} \chi(-\pi^*(l') + B + l)$.

We know by Theorem 5.1.1. from \cite{NNA2} that $q'$ is the minimal possible value of $h^1(\calO_{(Z''_{new, M}- B)_{r, M}}( \pi^*(l') - B))$ if we consider all analytic structures on the cycle $(Z''_{new, M}- B)_{r, M}$ (indeed the natural line bundle $\calO_{(Z''_{new, M}- B)_{r, M}}( \pi^*(l') - B)$ satisfies the negativity condition of coefficients).

It follows that $q \geq \chi(- \pi^*(l')) - \chi(- \pi^*(l') + B) + q'$, which means that $h^1(\calO_{(Z''_{new, M, p}- B)_{r, M}}( \pi^*(l') - B)) > q'$.

We also know that $h^1(\calO_{(Z''_{new, M, p}- B)_{r, M}}) < h^1(\calO_Z) = \alpha $, so by the induction hypothesis we know that there exists $p_1$ arbitrary close to $0$ in $Q_2$, such that $h^1(\calO_{(Z''_{new, M, p_1}- B)_{r, M}}(\pi^*(l') - B)) = h^1(\calO_{(Z''_{new, M, p}- B)_{r, M}}( \pi^*(l') - B)) - 1$.

Since the analytical type corresponding to $p_1 \in Q_2$ is a degeneration of the analytical type corresponding to $0 \in Q_2$ by semicontinuity of $h^1$ we have $h^1(\calO_{(Z''_{new, M, p_1}- l)_{r, M}}( \pi^*(l') - l)) \leq h^1(\calO_{(Z''_{new, M, p}- l)_{r, M}}( \pi^*(l') - l))$ for each $0 \leq l \leq  Z''_{new, M}$, so we get:

\begin{equation*}
\chi(- \pi^*(l')) - \min_{0 \leq l \leq Z''_{new, M, p_1}}( \chi(- \pi^*(l') + l)  - h^1(\calO_{(Z''_{new, M, p_1}- l)_{r, M}}( \pi^*(l') - l))) \leq k.
\end{equation*}

On the other hand we have $h^1(\calO_{(Z''_{new, M, p_1}- B)_{r, M}}(\pi^*(l') - B)) = h^1(\calO_{(Z''_{new, M, p}- B)_{r, M}}( \pi^*(l') - B)) - 1$, which yields:

\begin{equation*}
\chi(- \pi^*(l')) - \min_{0 \leq l \leq Z''_{new, M, p_1}}( \chi(- \pi^*(l') + l)  - h^1(\calO_{(Z''_{new, M, p_1}- l)_{r, M}}( \pi^*(l') - l))) \geq k-1.
\end{equation*}

Now if $G(p_1) = k-1$, then we are done with our claim with the choice of $p_2 = p_1$.

\underline{So we can assume that $G(p_1) = k$.}

In this case let's repeat the same procedure with $p_1 \in Q_2$ instead of $0 \in Q_2$, notice that $\sum_{0 \leq l \leq Z''_{new, M}} h^1(\calO_{(Z''_{new, M, p_1}- l)_{r, M}}(  \pi^*(l') - l)) < \sum_{0 \leq l \leq Z''_{new, M}} h^1(\calO_{(Z''_{new, M, p}- l)_{r, M}}(  \pi^*(l') - l))$ because of semicontinuity and the fact that $h^1(\calO_{(Z''_{new, M, p_1}- B)_{r, M}}( \pi^*(l') - B)) = h^1(\calO_{(Z''_{new, M, p}- B)_{r, M}}(  \pi^*(l') - B)) - 1$.

It means that we can repeat this procedure only finitely many times, and finally we get a point $p_2 \in Q_2$ arbitrary close to $0$, such that $G(p_2) = k-1$, this proves our lemma completely.

\end{proof}

\bigskip

We know that the deformation $\omegl_2 :\calz'_{new, M} \to (Q_{2}, 0)$ is complete in $q_2$.

We claim there are points $p'_2$ arbitrary close to $p_2$ such that:

\begin{equation*}
h^1( \calO_{Z''_{new, M, p'_2}}(\pi^*(l'))) = \chi(- \pi^*(l')) - \min_{0 \leq l \leq Z''_{new, M, p_2}}( \chi(- \pi^*(l') + l)  - h^1(\calO_{(Z''_{new, M, p_2}- l)_{r, M}}( \pi^*(l') - l))) = k-1.
\end{equation*}

Indeed we should fix the restriction of analytic structure $Z''_{new, M, p_2}$,  $Z''_{r, M, p_2}$ and we should move the analytic structure $Z''_{new, M, p_2}$ to get an analytic structure  $Z''_{new, M, p'_2}$ which is relatively generic to $Z''_{r, M, p_2}$.

By Theorem\ref{relgen2} this gives the desired analytic structure $p'_2 \in (Q_{2}, 0)$, the points $g'(p'_2) \in (Q, 0)$ proves our statement in the case $A = 0$, as it can be arbitrary close to $0$ in $Q$.

\subsection{Proof in the general case.}

Let's prove the statement in the following in the general case without the assumption $A = 0$.

We know that there are only finitely many possibilities for the fixed component cycle $A$ of the natural line bundle $\calO_{Z}(l') \in  \pic^{l'}(Z)$ for some analytic type, namely the cycles $0 \leq B \leq Z$.

Let's have the the number $m = \sum_{0 \leq B \leq Z, J \subset |Z-B|} h^0(\calO_{Z-B - E_J}( l' - B - E_J))$, we will prove our statement by induction on $m$, so assume that we know the statement for smaller values of $m$ (the case $m = 0$ is trivial).

Notice that the set of numbers $h^0(\calO_{Z-B - E_J}( l' - B- E_J)) | 0 \leq B \leq Z, J \subset |Z-B|)$ determine the fixed component cycle $A$ of the line bundle $\calO_{Z}(l')$.

Indeed $A$ is the minimal cycle such that $h^0(\calO_{Z-A - E_J}( l' - A - E_J)) < h^0(\calO_{Z-A}( l' - A))$.

Let's have a very large cycle $Z' > Z$, and let's have a complete deformation of the cycle $Z'$, namely $\omegl: \calz' \to (Q, 0)$, such that $\omegl^{-1}(0) = Z'$.

Assume that $A$ is the fixed component cycle of the line bundle $\calO_{Z}(l')$, which means that $A$ is the minimal cycle, such that $\calO_{Z-A}(l'-A) \in \im(c^{l'-A}(Z-A))$.

It means in particular that $k = h^1(\calO_{Z}(l')) = \chi(-l') - \chi(-l' + A) + h^1(\calO_{Z-A}(l'-A))$.

We know that

\begin{equation*}
\chi(-l') - \chi(-l' + A) + \chi(-l' + A) - \min_{0 \leq l \leq Z-A} \chi(-l' + A + l) \leq \chi(-l') - \min_{0 \leq l \leq Z} \chi(-l'+ l),
\end{equation*}
 which means that $h^1(\calO_{Z-A}(l'-A)) >\chi(-l' + A) - \min_{0 \leq l \leq Z-A} \chi(-l' + A + l)$.

\emph{Since $\calO_{Z-A}(l'-A) \in \im(c^{l'-A}(Z-A))$, we can apply the statement already proved in the special case $A = 0$:}

It means that there is a map $f: (\bC, 0) \to  (Q, 0)$, such that $f(0) = 0$ and for a generic element $t \in (\bC, 0)$, one has $h^1(\calO_{Z-A, f(t)}(l'-A)) =  h^1(\calO_{Z-A}(l'-A)) - 1 $.

Let's have now a generic element $t \in (\bC, 0) $ and let's denote $f(t) = p$, so we have $h^1(\calO_{Z-A, p}(l'-A)) =  h^1(\calO_{Z-A}(l'-A)) - 1$.

We have that $h^1(\calO_{Z, p}(l')) \geq \chi(-l') - \chi(-l' + A) + h^1(\calO_{Z-A, p}(l'-A)) = k-1$, on the other hand  by semicontinuity of $h^1$ we know that $h^1(\calO_{Z, p}(l')) \leq h^1(\calO_{Z}(l')) = k$.

If $h^1(\calO_{Z, p}(l'))  = k-1$, then we are done, so we can assume that $h^1(\calO_{Z, p}(l'))  = k$:

Notice that if $h^1(\calO_{Z, p}(l')) = k$, then we have $h^1(\calO_{Z, p}(l')) > \chi(-l') - \chi(-l' + A) + h^1(\calO_{Z-A, p}(l'-A)) = k-1$.
It means that $C \neq A$, where $C$ is the minimal cycle such that $\calO_{Z-C, p}(l'-C) \in \im( c^{l' - C}(Z-C))$.

On the other hand, this means that the fixed component cycle changes when passing to the analytic type $Z'_p$, which means that the set of numbers $(h^1(\calO_{Z- B - E_J, p}(l' -B - E_J)) | 0 \leq B \leq Z, J \subset |Z-B|)$ changes.

This yields that $\sum_{0 \leq B \leq Z, J \subset |Z-B|} \calO_{Z- B - E_J, p}(l' -B - E_J) < m$ and $ h^1(\calO_{Z, p}(l')) = k$, and so we are done by the induction hypothesis on $m$.

So we are done with the proof of the general case and it proves our theorem completely.

\end{proof}

Let's have finally the following corollary of the previous theorem:

\begin{cor}\label{geomgen}
Let $\mathcal{T}$ be an arbitrary rational homology sphere resolution graph with vertex set $\calv$ and let $Z$ be an effective cycle on it and let's have a singularity with resolution $\tX$ with resolution graph $\mathcal{T}$.

1) Suppose, that $k = h^1(\calO_Z) > 1 - \min_{E_{|Z|} \leq l \leq Z}  \chi(l)$, and let's have an arbitrary number $k > m \geq  1 - \min_{E_{|Z|} \leq l \leq Z}  \chi(l)$, then there is another singularity with resolution $\tX'$ supported on the resolution graph $\mathcal{T}$, for which one has $m = h^1(\calO_Z)$.

2) In particular this means that the possible values of the geometric genus $p_g(\tX)$ form an interval of integers.
\end{cor}
\begin{proof}
The first statement follows immediately from the previous theorem, since for each analytic structure $\tX$ supported on the resolution graph $\mathcal{T}$ one has
$h^1(\calO_Z) = h^1(\calO_{Z- E_{|Z|}}( -  E_{|Z|}))$, because of the following exact sequence:

\begin{equation}
0 \to  H^0(\calO_{Z- E_{|Z|}}( -  E_{|Z|})) \to  H^0(\calO_Z) \to H^0( \calO_{E_{|Z|}}) \to  H^1(\calO_{Z- E_{|Z|}}(-  E_{|Z|})) \to  H^1(\calO_Z) \to 0.
\end{equation}

Since the map $ H^0(\calO_Z) \to H^0( \calO_{E_{|Z|}}) $ is obviously surjective, we indeed get $h^1(\calO_Z) = h^1(\calO_{Z- E_{|Z|}}( -  E_{|Z|}))$ and then statement 1) follows
from Theorem\textbf{B}.

The second statement follows immediately if we apply the first statement for a very large cycle $Z$, since for a very large cycle $Z$ we have $h^1(\calO_Z) = p_g(\tX)$ for each analytic type $\tX$.
\end{proof}

\end{document}